\documentclass[reqno]{amsart}
\usepackage{amssymb,epsfig,amscd, verbatim}

\newtheorem{theorem}{Theorem}
\newtheorem{lemma}[theorem]{Lemma}
\newtheorem{proposition}[theorem]{Proposition}

\theoremstyle{definition}
\newtheorem{definition}[theorem]{Definition}
\newtheorem{example}[theorem]{Example}

\theoremstyle{remark}
\newtheorem{remark}[theorem]{Remark}

\numberwithin{equation}{section} \numberwithin{theorem}{section}
\newcommand{\thref}[1]{Theorem~{\rm\ref{#1}}}

\newcommand{\leref}[1]{Lemma~{\rm\ref{#1}}}

\newcommand{\deref}[1]{Definition~{\rm\ref{#1}}}

\newcommand\lbb[1]{\label{#1}}

\def\as{associative}
\def\conalg{conformal algebra}

\def\difalg{differential algebra}

\def\psalg{pseudo\-algebra}
\def\psalgs{pseudo\-algebras}

\def\tt{\otimes}                               
\def\<{\langle}
\def\>{\rangle}

\def\what{\widehat}

\def\d{\partial}

\def\tsum{{\textstyle\sum}}

\def\dsum{\displaystyle\sum}
\def\surjto{\twoheadrightarrow}                
\def\st{\; | \;}                               
\def\symm{S}                                   

\newcommand\xb[3]{[{#1}_{#2} {#3}]}
\newcommand\xp[3]{{#1}_{#2} {#3}}

\def\Zset{\mathbb{Z}}       

\def\Kset{{\mathbf{k}}}                

\newcommand{\kk}{\mathbf{k}}

\newcommand{\fg}{\mathfrak{g}}

\def\be{\beta}
\def\ga{\gamma}
\def\Ga{\Gamma}
\def\de{\delta}
\def\De{\Delta}
\def\ep{\varepsilon}
\def\la{\lambda}

\def\si{\sigma}

\def\g{{\mathfrak{g}}}      
\def\dd{{\mathfrak{d}}}

\def\A{{\mathcal{A}}}

\def\M{\mathcal{M}}

\def\O{\mathcal{O}}

\def\F{\mathcal{F}}           

\def\D{\mathcal{D}}           

\def\ue{U}                 
 
 \DeclareMathOperator{\codim}{codim}
\DeclareMathOperator{\Span}{span}

 \DeclareMathOperator{\id}{id}
 
\DeclareMathOperator{\gplk}{G}     
\DeclareMathOperator{\primt}{P}    
\DeclareMathOperator{\fil}{F}      
\DeclareMathOperator{\coh}{H}      

\DeclareMathOperator{\Hom}{Hom}
\DeclareMathOperator{\Lin}{Lin}     
\DeclareMathOperator{\End}{End} \DeclareMathOperator{\Chom}{Chom}
\DeclareMathOperator{\Cend}{Cend} 
\DeclareMathOperator{\gc}{gc}

 \DeclareMathOperator{\Cur}{Cur}

\begin{document}

\title{On  pseudo-bialgebras}

\author[C. Boyallian]{Carina Boyallian$^*$}
\author[J. Liberati]{Jos\'e I. Liberati$^*$}
\thanks{$^*$ Ciem-Famaf, Universidad Nacional de C\'ordoba,
Ciudad Universitaria, (5000) C\'ordoba, Argentina.  E-mail
addresses:  boyallia@mate.uncor.edu, joseliberati@gmail.com \hfill
\break \vskip .2cm \noindent{\bf Corresponding author: }Jos\'e I.
Liberati,
 \hfill \break
 Ciem-Famaf, Universidad Nacional de C\'ordoba,
Medina Allende y Haya de la Torre, Ciudad Universitaria, (5000)
C\'ordoba, Argentina. \hfill \break Tel: +54 (351) 4334051. Fax:
+54 (351) 4334054}


\begin{abstract}We  study  pseudoalgebras  from the  point of
 view of pseudo-dual of classical  Lie coalgebra
 structures. We define the notions
of Lie $H$-coalgebra and Lie pseudo-bialgebra. We obtain the
 analog of the CYBE, the Manin
triples and Drinfeld's double for Lie pseudo-bialgebras. We also
get a natural description of the annihilation algebra associated
to a pseudoalgebra as a convolution algebra, clarifying this
constructions in the theory of pseudoalgebras.
\end{abstract}

\maketitle

 \pagebreak

\section{Introduction}

The notion of conformal algebra was introduced by V. Kac as a
formal language describing the singular part of the operator
product expansion in two-dimensional conformal field theory (see
\cite{BKL}, \cite{BKV}, \cite{DK}, \cite{K2}, \cite{K3}, and
references there in).

In \cite{BDK}, Bakalov, D'Andrea and Kac develop a theory of
"multi-dimensional" Lie conformal algebras, called Lie
$H$-pseudoalgebras. Classification problems, cohomology theory and
representation theory have been developed (see  \cite{BDK},
\cite{BDK2}, \cite{BDK3}).

In the present work, we  study  Lie $H$-pseudoalgebras  from the
point of view of pseudo-dual of classical  Lie $H$-coalgebra
structures. We introduce the notions of Lie $H$-coalgebra and Lie
$H$-pseudo-bialgebra (see Section \ref{sec:pseudobialg}). In
Sections \ref{sec:cob}, \ref{submanin} and \ref{subdrinf}, we
obtain a pseudoalgebra analog of the CYBE, we study coboundary Lie
$H$-pseudo-bialgebras, and a pseudoalgebra version of Manin
triples and Drinfeld's double. Usually, in the theory of Lie
$H$-pseudoalgebras the proofs of pseudoalgebra version of
classical results need to be carefully translated, as in the
present work.

Two Lie algebras are usually associated to a Lie $H$-pseudoalgebra
$L$, that is $\A_Y(L)$ and the annihilation algebra (see \cite{K2}
and Section 7 in \cite{BDK}). Their construction, at first sight,
is not natural unless one look at a similar notion from vertex
algebra theory. In Section \ref{subanni}, using the language of
Lie $H$-coalgebras, we will see them as convolution algebras of
certain type, obtaining a more natural and conceptual construction
of them.

This work is a  generalization to the language of Lie
$H$-pseudoalgebras of the results obtained in \cite{L1}.

\vskip 1cm

\section{Preliminaries on Hopf Algebras}\lbb{sprelh}


\vskip .5cm

Unless otherwise specified, all vector spaces, linear maps and
tensor products are considered over an algebraically closed field
$\kk$ of characteristic 0.

 In  this section we present some facts
and notation which will be used throughout the paper. The material
in Sections \ref{subnotid} and \ref{subfiltop} is standard and can
be found, for example, in Sweedler's book \cite{Sw}. The material
in Section~\ref{subfour} is taken from \cite{BDK}.

\subsection{Notation and basic identities}\lbb{subnotid}
Let $H$ be a
Hopf algebra with a coproduct $\De$, a counit $\ep$, and an
antipode $S$.

We will use the standard Sweedler's notation (cf.\ \cite{Sw}):
\begin{align}
\lbb{de1} \De(h) &= h_{(1)} \tt h_{(2)},
\\
\lbb{de2} (\De\tt\id)\De(h) &= (\id\tt\De)\De(h) = h_{(1)} \tt
h_{(2)} \tt h_{(3)},
\\
\lbb{de3} (S\tt\id)\De(h) &= h_{(-1)} \tt h_{(2)},
\quad\text{etc.}
\end{align}
Observe that notation \eqref{de2} uses the coassociativity of
$\De$. The axioms of the antipode and the counit can be written as
follows:
\begin{align}
\lbb{antip} h_{(-1)} h_{(2)} &= h_{(1)} h_{(-2)} = \ep(h),
\\
\lbb{cou} \ep(h_{(1)}) h_{(2)} &= h_{(1)} \ep(h_{(2)}) = h,
\end{align}
while the fact that $\De$ is a homomorphism of algebras translates
as:
\begin{equation*}
\lbb{deprod} (fg)_{(1)} \tt (fg)_{(2)} = f_{(1)} g_{(1)} \tt
f_{(2)} g_{(2)}.
\end{equation*}
Equations \eqref{antip} and \eqref{cou} imply the following useful
relations:
\begin{equation}
\lbb{cou2} h_{(-1)} h_{(2)} \tt h_{(3)} = 1\tt h = h_{(1)}
h_{(-2)} \tt h_{(3)}.
\end{equation}

\

Since we shall work with cocommutative Hopf algebras, we recall
the following important and classical result (for the proof see
\cite{Sw}):

\begin{theorem}[Kostant]\lbb{tkostant}
Let $H$ be a cocommutative Hopf algebra over $\Kset$ $($an
algebraically closed field of characteristic $0)$. Then $H$ is
isomorphic $($as a Hopf algebra$)$ to the smash product of the
universal enveloping algebra $\ue(\primt(H))$ and the group
algebra $\Kset[\gplk(H)]$, where $\gplk(H)=\{\, g\in H\, | \,
\De(g)=g\tt g\}$ is the subset of group-like elements of $H$, and
$\primt(H) = \{\,p\in H\, |\, \De(p)=p\tt 1 + 1\tt p\}$ is the
subspace of primitive elements of $H$.
\end{theorem}

An \as\ algebra $A$ is called an {\em $H$-\difalg\/} if it is also
a left $H$-module such that the multiplication $A\tt A\to A$ is a
homomorphism of $H$-modules. That is,
\begin{equation}\lbb{hxy}
h(xy) = (h_{(1)}x) (h_{(2)}y)
\end{equation}
for $h\in H$, $x,y\in A$. Observe that $H$ itself is an
$H$-bimodule, however $H$ is not an $H$-\difalg{}.\

\

Another important property  of a cocommutative Hopf algebra is
that the antipode is an involution, i.e., $S^2 = \id$, which will
be convenient in allowing us not to distinguish between $S$ and
its inverse.


\subsection{Filtration and topology}\lbb{subfiltop}
We consider an increasing 
sequence of subspaces of a Hopf algebra $H$ defined inductively
by:
\begin{align*}
\fil^n H &=0 \;\;\text{for $n<0$,} \quad \fil^0 H =
\Kset[\gplk(H)],
\\
\fil^n H &= \Span_\kk \big\{ h\in H \;\big|\; \De(h)\in \fil^0
H\tt h + h\tt\fil^0 H
\\
&\qquad\qquad\qquad\qquad + \tsum_{i=1}^{n-1} \fil^i
H\tt\fil^{n-i} H \big\}, \qquad n\geq 1.\nonumber
\end{align*}
It has the following properties (which are immediate from
definitions):
\begin{align*}
(\fil^m H) (\fil^n H) &\subset \fil^{m+n} H,
\\
\De(\fil^n H) &\subset \tsum_{i=0}^n \fil^i H\tt\fil^{n-i} H,
\\
S(\fil^n H) &\subset \fil^n H.
\end{align*}
If $H$ is cocommutative, using \thref{tkostant}, one can show
that:
\begin{equation}\lbb{filh0}
\bigcup_n \fil^n H = H.
\end{equation}
(This condition is also satisfied when $H$ is a quantum universal
enveloping algebra). Provided that \eqref{filh0} holds, we say
that a nonzero element $a\in H$ has degree $n$ if $a \in \fil^n H
\setminus \fil^{n-1} H$.

When $H$ is a universal enveloping algebra, we get its canonical
filtration. Later in some instances we will also impose the
following finiteness condition on $H$:
\begin{equation}\lbb{filh4}
\dim \fil^n H < \infty \qquad \forall n.
\end{equation}
It is satisfied when $H$ is a universal enveloping algebra of a
finite-dimensional Lie algebra, or its smash product with the
group algebra of a finite group.

Now let $X=H^* := \Hom_\Kset(H,\Kset)$ be the dual of $H$. It
inherits a multiplication defined as the dual of the
comultiplication in $H$.  Recall that $H$ acts (on the left) on
$X$ by the formula ($h,f\in H$, $x\in X$):
\begin{equation}\lbb{hx}
\langle hx, f\rangle = \langle x, S(h)f\rangle.
\end{equation}
Then, since
\begin{equation*}\lbb{dif-alg-X}
h(xy) = (h_{(1)}x ) (h_{(2)}y ),
\end{equation*}
we have that $X$ is an \as\ $H$-\difalg\ (see \eqref{hxy}).
Moreover, $X$ is commutative when $H$ is cocommutative. Similarly,
one can define a right action of $H$ on $X$ by
\begin{equation}\lbb{xh}
\langle xh, f\rangle = \langle x, f S(h)\rangle,
\end{equation}
and then we have
\begin{equation}\lbb{xyh}
(xy)h = (x h_{(1)}) (y h_{(2)}).
\end{equation}
Observe that associativity of $H$ implies that $X$ is an
$H$-bimodule, i.e.
\begin{equation}\lbb{fxg}
f(xg) = (fx)g, \qquad f,g\in H, \; x\in X.
\end{equation}

Let $X=\fil_{-1}X\supset\fil_0X\supset\dotsm$ be the decreasing
sequence of subspaces of $X$ dual to $\fil^n H$, namely
$$
\fil_n X = (\fil^n H)^\perp=\{\, x\in X\,|\, \langle x, f\rangle
=0, \hbox{ for all } f\in \fil^n H\,\}.
$$
It has the following properties:
\begin{align}
\lbb{fils1} (\fil_m X) (\fil_n X) &\subset \fil_{m+n} X,
\\
\lbb{fils5} (\fil^m H) (\fil_n X) &\subset \fil_{n-m} X,
\end{align}
and
\begin{equation}\lbb{fils0}
\bigcap_n \fil_n X = 0, \quad\text{provided that \eqref{filh0}
holds}.
\end{equation}
We define a topology of $X$ by considering $\{\fil_n X\}$ as a
fundamental system of neighborhoods of $0$. We will always
consider $X$ with this topology, while $H$ with the discrete
topology. It follows from \eqref{fils0} that $X$ is Hausdorff,
provided that \eqref{filh0} holds. By \eqref{fils1} and
\eqref{fils5}, the multiplication of $X$ and the action of $H$ on
it are continuous; in other words, $X$ is a topological
$H$-\difalg.

Now, we define an antipode $S\colon X\to X$ as the dual of that of
$H$:
\begin{equation}\lbb{sxhxsh}
\langle S(x), h \rangle = \langle x, S(h) \rangle .
\end{equation}
Then we have:
\begin{equation}\lbb{sabsbsa}
S(ab) = S(b) S(a) \qquad\text{for}\quad a,b \in X \;\text{or}\; H.
\end{equation}

We will also define a comultiplication $\De\colon X\to
X\widehat{\tt}X$ as the dual of the multiplication $H\tt H\to H$,
where $X\widehat{\tt}X := (H\tt H)^*$ is the completed tensor
product. Formally, we will use the same notation for $X$ as for
$H$ (see \eqref{de1}--\eqref{de3}), writing for example $\De(x) =
x_{(1)} \tt x_{(2)}$ for $x\in X$. By definition, for $x,y \in X$,
$f,g\in H$, we have:
\begin{align}
\lbb{xy,f} \langle xy,f \rangle &= \langle x\tt y, \De(f) \rangle
= \langle x,f_{(1)} \rangle \langle y,f_{(2)} \rangle ,
\\
\lbb{x,fg} \langle x,fg \rangle &= \langle \De(x), f\tt g \rangle
= \langle x_{(1)}, f \rangle \langle x_{(2)}, g \rangle.
\end{align}
We have:
\begin{align}
\lbb{fils3} S(\fil_n X) &\subset \fil_n X,
\\
\lbb{fils2} \De(\fil_n X) &\subset \dsum_{i=-1}^n \fil_i
X\what\tt\fil_{n-i} X.
\end{align}
If $H$ satisfies the finiteness condition \eqref{filh4}, then the
filtration of $X$ satisfies
\begin{equation}\lbb{fils4}
\dim X / \fil_n X < \infty \qquad \forall n,
\end{equation}
which implies that $X$ is linearly compact (see Sect.6 in
\cite{BDK} for details).

By a basis of $X$ we will always mean a topological basis
$\{x_i\}$ which tends to $0$, i.e., such that for any $n$ all but
a finite number of $x_i$ belong to $\fil_n X$. Let $\{h_i\}$ be a
basis of $H$ (as a vector space) compatible with the increasing
filtration. Then the set of elements $\{x_i\}$ of $X$ defined by
$\langle x_i, h_j \rangle = \de_{ij}$ is called the dual basis of
$X$. If $H$ satisfies \eqref{filh4}, then $\{x_i\}$ is a basis of
$X$ in the above sense, i.e., it tends to $0$. We have for $g\in
H$, $y\in X$:
\begin{equation}\label{AAA}
g = \tsum_i \, \langle g, x_i \rangle h_i, \quad y = \tsum_i \,
\langle y, h_i \rangle x_i,
\end{equation}
where the first sum is finite, and the second one is convergent in
$X$.


\

\begin{example}\lbb{eued}
Let $H=\ue(\dd)$ be the universal enveloping algebra of an
$N$-dimension\-al Lie algebra $\dd$. Fix a basis $\{\d_i\}$ of
$\dd$, and for $I = (i_1,\dots,i_N) \in\Zset_+^N$ let $\d^{(I)} =
\d_1^{i_1} \dotsm \d_N^{i_N} / i_1! \dotsm i_N!$. Then
$\{\d^{(I)}\}$ is a basis of $H$ (the Poincar\'e--Birkhoff--Witt
basis). Moreover, it is easy to see that
\begin{equation}\lbb{dedn}
\De(\d^{(I)}) = \sum_{ J+K=I } \d^{(J)} \tt \d^{(K)}.
\end{equation}
If $\{t_I\}$ is the dual basis of $X$, defined by $\langle t_I,
\d^{(J)} \rangle = \de_{I,J}$, then \eqref{dedn} implies $t_J t_K
= t_{ J+K }$. Therefore, $X$ can be identified with the ring $\O_N
= \Kset[[t_1,\dots,t_N]]$ of formal power series in $N$
indeterminates. Then the action of $H$ on $\O_N$ is given by
differential operators.
\end{example}

\


\subsection{Fourier transform}\lbb{subfour} (cf. \cite{BDK})
For an arbitrary Hopf algebra $H$, we introduce a map $\F\colon
H\tt H \to H\tt H$, called the {\em Fourier transform}, by the
formula
\begin{equation}
\lbb{ftrans} \F(f\tt g) = (f\tt 1) \, (S\tt\id)\De(g) = f g_{(-1)}
\tt g_{(2)}.
\end{equation}
Observe  that $\F$ is a vector space isomorphism with an inverse
given by
\begin{equation*}
\lbb{finv} \F^{-1}(f\tt g) = (f\tt 1) \,\De(g) = f g_{(1)} \tt
g_{(2)},
\end{equation*}
since, using the coassociativity of $\De$ and \eqref{cou2}, we
have
\begin{displaymath}
\F^{-1}(f g_{(-1)} \tt g_{(2)}) = f g_{(-1)} (g_{(2)})_{(1)} \tt
(g_{(2)})_{(2)} = f g_{(-1)} g_{(2)} \tt g_{(3)} = f\tt g.
\end{displaymath}
The significance of $\F$ is in the identity
\begin{equation}\lbb{sigfour}
f\tt g = \F^{-1}\F(f\tt g) = (f g_{(-1)} \tt 1) \, \De(g_{(2)}),
\end{equation}
which,  implies the next result.

\begin{lemma}\cite{BDK} \lbb{lhhh}
(a) If  $\{h_i\}$, $\{x_i\}$ are dual bases in $H$ and $X$, then
\begin{equation}\lbb{xshxy}
\De(x) = \dsum_i \, xS(h_i) \tt x_i
       = \dsum_i \, x_i \tt S(h_i) x
\end{equation}
for any $x\in X$.

\noindent (b) Every element of $H\tt H$ can be uniquely
represented in the form $\sum_i \, (h_i \tt 1)\De(l_i)$, where
$\{h_i\}$ is a fixed $\Kset$-basis of $H$ and $l_i\in H$. In other
words, $H\tt H = (H\tt\Kset)\De(H)$.

\noindent (c)
\begin{equation}\label{identity} (1\tt g\tt 1)\tt_H 1= (g_{(-1)}\tt 1 \tt
g_{(-2)})\tt_H 1\in (H\tt H\tt H)\tt_H \mathbf{k}
\end{equation}
for all $g\in H$.
\end{lemma}




\vskip 1cm

\section{Lie $H$-Pseudoalgebras}\lbb{slie*}

\vskip .5cm

 The notion of conformal algebra \cite{K2} was
generalized by the notion of Lie $H$-pseudoalgebra in \cite{BDK}.
They can be considered as Lie algebras in a certain
``pseudotensor'' category, instead of the category of vector
spaces. A pseudotensor category \cite{BD} is a category equipped
with ``polylinear maps'' and a way to compose them (such
categories were first introduced by Lambek \cite{L} under the name
multicategories). This is enough to define the notions of Lie
algebra, representations, cohomology, etc.

In this section, we shall recall the example of pseudotensor
category that will be used. We follow the exposition in
\cite{BDK}. The proofs of all the statements in this section can
be found in \cite{BDK}. Let $H$ be a cocommutative Hopf algebra
with a comultiplication $\De$. We introduce a pseudotensor
category $\M^*(H)$ whose objects are the same objects as in
$\M^l(H)$ (the category of left $H$-modules), but with a
non-trivial pseudotensor structure \cite{BD}. More precisely, the
space of polylinear maps from $\{L_i\}_i\in I$ to $M$ is defined
by ($L_i,M\in \M^l(H)$, and $I$ a finite non-empty set)
\begin{equation}\lbb{m*d}
\Lin(\{L_i\}_{i\in I}, M) = \Hom_{H^{\tt I}} \left(\ \underset
{i\in I}{\boxtimes}\  L_i\ ,\  H^{\tt I}\tt_H M \ \right),
\end{equation}
where $\boxtimes_{i\in I}$ is the tensor product functor
$\M^l(H)^I \to \M^l(H^{\tt I})$. For any surjection $\pi\colon
J\surjto I$, and  a collection  $\{N_j\}_j\in J$ of objects in
$\M^l(H)$ we define the composition of polylinear maps as follows:

\begin{align}
&\Lin(\{L_i\}_{i\in I}, M)\tt \tt_{i\in I} \Lin(\{N_j\}_{j\in
J_i}, L_i)\longrightarrow \Lin(\{N_j\}_{j\in J}, M) \nonumber\\
\lbb{*com} & \qquad \phi\,\, \times \,\, \{\psi_i\}_{i\in I}
    \longmapsto \phi \circ \{\psi_i\}_{i\in I}
    := \De^{(\pi)}\bigl(\phi\bigr)
\, \circ\bigl(\boxtimes_{i\in I}\psi_i\bigr),
\end{align}
\vskip .2cm

\noindent where $J_i=\pi^{-1}(i)$ for $i\in I$ and  $\De^{(\pi)}$
is the functor $\M^l(H^{\tt I})\to\M^l(H^{\tt J})$, $M\mapsto
H^{\tt J} \tt_{H^{\tt I}} M$, where $H^{\tt I}$ acts on $H^{\tt
J}$ via the iterated comultiplication determined by $\pi$.

Explicitly, let $n_j \in N_j$ $(j\in J)$, and write
\begin{equation}\lbb{compoly1}
\psi_i\Big( \underset{j\in J_i}{\tt} n_j \Big) = \sum_r g_i^r
\tt_H l_i^r, \qquad g_i^r \in H^{\tt J_i}, \; l_i^r \in L_i,
\end{equation}
where, as before, $J_i = \pi^{-1}(i)$ for $i\in I$. Let
\begin{equation}\lbb{compoly2}
\phi\Bigl(  \underset{i\in I}{\tt} l_i^r \Bigr) = \sum_s f^{rs}
\tt_H m^{rs}, \qquad f^{rs} \in H^{\tt I}, \; m^{rs} \in M.
\end{equation}
Then, by definition,
\begin{equation}\lbb{compoly3}
\bigl(\phi\ \circ\ \{\psi_i\}_{i\in I}\bigr) \Bigl( \
\underset{j\in J}{\tt} n_j \Bigr) = \sum_{r,s} \Big(
\underset{i\in I}{\tt} g_i^r \Big) \De^{(\pi)}(f^{rs}) \tt_H
m^{rs},
\end{equation}
where $\De^{(\pi)} \colon H^{\tt I} \to H^{\tt J}$ is the iterated
comultiplication determined by $\pi$. For example, if
$\pi\colon\{1,2,3\}\to\{1,2\}$ is given by $\pi(1)=\pi(2)=1$,
$\pi(3)=2$, then $\De^{(\pi)} = \De\tt\id$; if $\pi(1)=1$,
$\pi(2)=\pi(3)=2$, then $\De^{(\pi)} = \id\tt\De$.

The symmetric group $\symm_I$ acts among the spaces
$\Lin(\{L_i\}_{i\in I}, M)$ by simultaneously permuting the
factors in $\boxtimes_{i\in I} L_i$ and $H^{\tt I}$. This is the
only place where we need the cocommutativity of $H$; for example,
the permutation $\sigma_{12} = (12) \in\symm_2$ acts on $(H\tt
H)\tt_H M$ by
\begin{displaymath}
\si_{12}\bigl( (f\tt g)\tt_H m \bigr) = (g\tt f)\tt_H m,
\end{displaymath}
and this is well defined only when $H$ is cocommutative.

There is a  generalization of  the above construction for
quasitriangular Hopf algebras that will not be used in this
sequel, see Remark 3.4 in \cite{BDK} for details.


We introduce the following notion.

\begin{definition}\lbb{dpseu*}
An {\em{$H$-\psalg\/}} (or just a \psalg) is an object $A$ in
$\M^*(H)$ (i.e. a left $H$-module) together with an operation
$\mu\in \Lin(\{A,A\},A)= \Hom_{H\tt H}(A\tt A, (H\tt H)\tt_H A)$,
called the {\em{pseudoproduct}}.
\end{definition}

Equivalently,  an $H$-pseudoalgebra is a left $H$-module $A$
together with a map
\begin{align*}
& A\tt A \longrightarrow (H\tt H)\tt_H A \\
& a\tt b \longmapsto a*b=\mu(a\tt b)
\end{align*}
satisfying  the following defining property: \vskip .3cm

{\bf $H$-bilinearity:} For $a,b\in A$, $f,g\in H$, one has
\begin{equation}\lbb{bil*2}
fa*gb = ((f\tt g)\tt_H 1) \, (a*b).
\end{equation}

\noindent That is, if
\begin{equation}\lbb{ab*'}
a*b = \dsum_i\, (f_i\tt g_i)\tt_H e_i,
\end{equation}
\noindent then $fa*gb = \tsum_i\, (f f_i\tt g g_i)\tt_H e_i$.

\

\begin{definition}\lbb{dflie*}
A {\em{Lie $H$-\psalg\/}} (or just a Lie \psalg) is a Lie algebra
$(L,\mu)$ (with $\mu\in \Lin(\{L,L\},L)$) in the pseudotensor
category $\M^*(H)$ as defined above.
\end{definition}

%
%

In order to give an  explicit and equivalent definition of a Lie
$H$-pseudoalgebra we need to compute the compositions
$\mu(\mu(\cdot,\cdot),\cdot)$ and $\mu(\cdot,\mu(\cdot,\cdot))$ in
$\M^*(H)$. Let $a*b$ be given by \eqref{ab*'}, and let
\begin{equation}\lbb{abc*1'}
e_i*c = \dsum_{i,j}\, (f_{ij}\tt g_{ij})\tt_H e_{ij}.
\end{equation}
Then $(a*b)*c \equiv \mu(\mu(a\tt b) \tt c)$ is the following
element of $H^{\tt3}\tt_H A$ (cf.\ \eqref{compoly3}):
\begin{equation}\lbb{abc*3'}
(a*b)*c = \dsum_{i,j}\, (f_i \, {f_{ij}}_{(1)}\tt g_i\,
{f_{ij}}_{(2)}\tt g_{ij}) \tt_H e_{ij}.
\end{equation}
Similarly, if we write
\begin{align}
\lbb{abc*4'} b*c &= \dsum_i\, (h_i\tt l_i)\tt_H d_i,
\\
\lbb{abc*5'} a*d_i &= \dsum_{i,j}\, (h_{ij}\tt l_{ij})\tt_H
d_{ij}, \intertext{then} \lbb{abc*6'} a*(b*c) &= \dsum_{i,j}\,
(h_{ij}\tt h_i \,{l_{ij}}_{(1)}\tt l_i\, {l_{ij}}_{(2)}) \tt_H
d_{ij}.
\end{align}


Equivalent definition:  a {\it Lie pseudoalgebra}  is a left
$H$-module  $L$ endowed with a map,
\begin{displaymath}
L\otimes L  \longrightarrow (H\tt H)\tt_H L, \qquad a\otimes b
\mapsto [a * b]
\end{displaymath}
called the {\it pseudobracket}, and  satisfying the following
axioms $(a,\, b,\, c\in
 L; f,g\in H)$,

\begin{description}

\item[$H$-bilinearity]
\begin{equation}\lbb{bil*}
[fa*gb] = ((f\tt g)\tt_H 1) \, [a*b].
\end{equation}

\item[Skew-commutativity]
\begin{equation}\lbb{ssym*}
[b*a] = - (\si\tt_H\id) \, [a*b].
\end{equation}
\end{description}
where $\si\colon H\tt H\to H\tt H$ is the permutation $\si(f\tt
g)=g\tt f$. Explicitly, $[b*a] = - \tsum_i\, (g_i\tt f_i)\tt_H
e_i$, if $[a*b]=\sum_i (f_i\tt g_i)\tt_H e_i$ . Note that the
right-hand side of \eqref{ssym*} is well defined due to the
cocommutativity of $H$.

\begin{description}
\item[Jacobi identity]
\begin{equation}\lbb{jac*}
[a*[b*c]] - ((\si\tt\id)\tt_H\id) \, [b*[a*c]] = [[a*b]*c]
\end{equation}
\end{description}
in $H^{\tt3}\tt_H L$, where the compositions $[[a*b]*c]$ and
$[a*[b*c]]$ are defined as above.

\


One can also define {\em{\as\ $H$-\psalgs\/}} as associative
algebras $(A,\mu)$ in the pseudotensor category $\M^*(H)$. More
precisely,  a pseudoproduct $a*b$ is {\it associative} iff it
satisfies
\begin{description}
\item[Associativity]
\begin{equation}\lbb{assoc*}
a*(b*c) = (a*b)*c
\end{equation}
\end{description}
in $H^{\tt3}\tt_H A$, where the compositions $(a*b)*c$ and
$a*(b*c)$ are given by the above formulas.

\

Similarly, the pseudoproduct $a*b$ is commutative iff it satisfies

\begin{description}
\item[Commutativity]
\begin{equation}\lbb{comm*}
b*a = (\si\tt_H\id) \, (a*b),
\end{equation}

\end{description}

\


Given $(A,\mu)$  an \as\ $H$-\psalg, one can define a
pseudobracket $\be$ as the commutator
\begin{equation}\label{eq:commutator}
[a*b] = a*b - (\si\tt_H\id) \, (b*a).
\end{equation}
 Then, it is easy to check that $(A,\be)$ is a Lie $H$-\psalg.


The definitions of representations of Lie \psalgs\ or \as\ \psalgs
\  are obvious modifications of the usual one. For example,

\begin{definition}\lbb{dreplie*}
A {\em representation\/} of a Lie $H$-\psalg\ $L$ is a left
$H$-module $M$ together with an operation $\rho\in
\Lin(\{L,M\},M)$, that we denote by $a*c \equiv \rho(a\tt c)$,
which satisfies
\begin{equation}\lbb{replie*}
a*(b*c) - ((\si\tt\id)\tt_H\id) \, (b*(a*c)) = [a*b]*c
\end{equation}
for $a,b\in L$, $c\in M$.
\end{definition}

\vskip .3cm

\begin{example}The (Lie) conformal algebras introduced by Kac \cite{K2} are
exactly the (Lie) $\Kset[\d]$-\psalgs, where $\Kset[\d]$ is the
Hopf algebra of polynomials in one variable $\d$. The explicit
relation between the $\la$-bracket of \cite{DK} and the
pseudobracket  is:
\begin{displaymath}\lbb{alab}
[a_\la b] = \tsum_i\, p_i(\la)c_i \;\;\iff\;\; [a*b] = \tsum_i\,
(p_i(-\d)\tt1)\tt_{\Kset[\d]} c_i.
\end{displaymath}

Similarly, for $H=\Kset[\d_1,\dots,\d_N]$ we get conformal
algebras in $N$ indeterminates, see \cite[Section~10]{BKV}. We may
say that for $N=0$, $H$ is $\Kset$; then a $\Kset$-conformal
algebra is the same as a Lie algebra.

On the other hand, when $H=\Kset[\Ga]$ is the group algebra of a
group $\Ga$, one obtains the $\Ga$-conformal algebras studied in
\cite{GK}.
\end{example}
\begin{example}{\it Current \psalgs}.\lbb{subcuralg}
Let $H'$ be a Hopf subalgebra of $H$, and let $A$ be an
$H'$-\psalg. Then we define the {\em current\/} $H$-\psalg\
$\Cur_{H'}^{H} A \equiv\Cur A$ as ${H}\tt_{H'} A$ by extending the
pseudoproduct $a*b$ of $A$ using the $H$-bilinearity. Explicitly,
for $a,b\in A$ and $f,g\in H$, we define
\begin{displaymath}\lbb{curalg*}
\begin{split}
(f\tt_{H'}a) * (g\tt_{H'}b) &= ((f\tt g)\tt_{H}1) \, (a*b)
\\
&= \tsum_i\, (f f_i \tt g g_i)\tt_{H} (1 \tt_{H'} e_i),
\end{split}
\end{displaymath}
if $a*b = \tsum_i\, (f_i \tt g_i)\tt_{H'} e_i$.
%
%
Then $\Cur_{H'}^{H} A$ is an $H$-\psalg\ which is Lie or \as\ when
$A$ is so.

An important special case is when $H'=\Kset$: given a Lie algebra
$\g$, let $\Cur\g=H\tt\g$ with the following pseudobracket
\begin{displaymath}\lbb{curalg2}
[(f \tt a) * (g \tt b)] = (f\tt g)\tt_{H} (1\tt [a,b]).
\end{displaymath}
Then $\Cur\g$ is a Lie $H$-\psalg.
 \end{example}


 Now we will  introduce the notion of  $x$-products in order
  to reformulate the definition
of a Lie (or \as) $H$-\psalg\ in terms of the properties of the
$x$-brackets (or products). In this way, we obtain an algebraic
structure equivalent to that of an $H$-\psalg, that is called an
$H$-\conalg. This formulation is analogous to the $(n)$-products
in the setting of conformal algebras \cite{K3}, and we shall use
it in the following section.

Let $(L,[\,*\,])$ be a Lie $H$-\psalg. Recall the Fourier
transform $\F$, defined by \eqref{ftrans}:
\begin{displaymath}
\F(f\tt g) = f g_{(-1)} \tt g_{(2)} ,
\end{displaymath}
and the identity \eqref{sigfour}:
\begin{displaymath}
f\tt g = (f g_{(-1)} \tt 1) \, \De(g_{(2)}).
\end{displaymath}
Using them, for any $a,b\in L$, we have that $[a*b] = \tsum_i\,
(f_i\tt g_i) \tt_H e_i$ can be rewritten as
\begin{equation}\lbb{psbr2}
[a*b] = \dsum_i\, (f_i {g_i}_{(-1)}  \tt 1) \tt_H {g_i}_{(2)} e_i.
\end{equation}
Hence $[a*b]$ can be written uniquely in the form $\tsum_i\, (h_i
\tt 1) \tt_H c_i$, where $\{h_i\}$ is a fixed $\Kset$-basis of $H$
(cf.\ \leref{lhhh}).

Now, we introduce another bracket $[a,b]\in H\tt L$ defined as the
Fourier transform of $[a*b]$:
\begin{equation*}\lbb{bracket}
[a,b] = \dsum_i\, \F(f_i\tt g_i) \, (1\tt e_i) = \dsum_i\, f_i \,
{g_i}_{(-1)} \tt {g_i}_{(2)} e_i.
\end{equation*}
That is,
\begin{equation}\lbb{ab}
[a,b] = \tsum_i\, h_i\tt c_i \quad\text{if}\quad [a*b] = \tsum_i\,
(h_i\tt1)\tt_H c_i.
\end{equation}
\vskip .3cm

\noindent Then for $x\in X=H^*$,  we define the {\em
$x$-bracket\/} in $L$ as follows:
\begin{align}\lbb{xbracket}
\xb{a}{x}{b}
: &= (\langle S(x),\cdot\rangle\tt\id) \, [a,b] \\&= \dsum_i\, \<
S(x), f_i \,{g_i}_{(-1)} \> \, {g_i}_{(2)} \, e_i= \dsum_i\,
\langle S(x), h_i \rangle c_i,\nonumber
\end{align}
if $ [a*b] = \sum_i\, (f_i\tt g_i) \tt_H e_i= \tsum_i\,
(h_i\tt1)\tt_H c_i$.

\

Using properties of the Fourier transform, it is straightforward
to derive the properties of the bracket \eqref{ab}. Then the
definition of a Lie \psalg\ can be equivalently reformulated as
follows.

\begin{definition}\lbb{dhconf1}
A {\em Lie $H$-conformal algebra\/} is a left $H$-module $L$
equipped with a bracket $[\cdot,\cdot]\colon L\tt L\to H\tt L$,
satisfying the following properties ($a,b,c\in L$, $h\in H$):
\begin{description}

\item[$H$-sesqui-linearity]
\begin{align*}
[ha,b] &= (h\tt1) \, [a,b],
\\
[a,hb] &= (1\tt h_{(2)}) \, [a,b] \, (h_{(-1)}\tt1).
\end{align*}

\item[Skew-commutativity]
If $[a,b]$ is given by \eqref{ab}, then
\begin{equation}\lbb{ss}
[b,a] = - \dsum_i\, {h_i}_{(-1)} \tt {h_i}_{(2)} c_i.
\end{equation}

\item[Jacobi identity]
\begin{equation}\lbb{jacid}
[a,[b,c]] - (\si\tt\id) \, [b,[a,c]] = (\F^{-1}\tt\id) \,
[[a,b],c]
\end{equation}
\end{description}
\vskip .2cm
in $H\tt H\tt L$, where $\si\colon H\tt H\to H\tt H$ is the
permutation $\si(f\tt g)=g\tt f$, and
\begin{align*}
[a,[b,c]] &= (\si\tt\id) \, (\id\tt[a,\cdot]) \, [b,c],
\\
[[a,b],c] &= (\id\tt[\cdot,c]) \, [a,b].
\end{align*}

\end{definition}


One can also reformulate \deref{dhconf1} in terms of the
$x$-brackets \eqref{xbracket}.

\begin{definition}\lbb{dhconf2}
A {\em Lie $H$-conformal algebra\/} is a left $H$-module $L$
equipped with $x$-brackets $\xb axb \in L$ for $a,b\in L$, $x\in
X$, satisfying the following properties:
\begin{description}

\item[Locality]
\begin{equation}\lbb{codimx}
\codim\{ x\in X \st \xb{a}{x}{b} = 0\} < \infty \quad\text{for any
$a,b\in L$}.
\end{equation}

\end{description}
Equivalently, for any basis $\{x_i\}$ of $X$,
\begin{equation*}\lbb{loc2}
\xb{a}{x_i}{b} \ne 0 \quad\text{for only a finite number of $i$.}
\end{equation*}

\begin{description}
\item[$H$-sesqui-linearity]
\begin{align}
\lbb{bl1} \xb{ha}{x}{b} &= \xb{a}{xh}{b},
\\
\lbb{bl2} \xb{a}{x}{hb} &= h_{(2)} \, \xb{a}{h_{(-1)} x}{b}.
\end{align}

\item[Skew-commutativity]
Choose dual bases $\{h_i\}$, $\{x_i\}$ in $H$ and $X$. Then:
\begin{equation}\lbb{sksy}
\xb{a}{x}{b} = -\dsum_i \langle x, {h_i}_{(-1)} \rangle \,
{h_i}_{(-2)} \, \xb{b}{x_i}{a}.
\end{equation}

\item[Jacobi identity]
\begin{equation}\lbb{jid}
\xb{a}{x}{\xb{b}{y}{c}} - \xb{b}{y}{\xb{a}{x}{c}} = [[a_{x_{(2)}}
b]_{y x_{(1)}} c].
\end{equation}

\end{description}

\end{definition}


 \vskip 1cm

We need the following important notions (see Section 10 in
\cite{BDK}).


\begin{definition}\lbb{dchom}
Let $V$ and $W$ be two $H$-modules. An {\em $H$-pseudolinear
map\/} from $V$ to $W$ is a $\Kset$-linear map $\phi\colon V\to
(H\tt H)\tt_H W$ such that
\begin{align}
\lbb{hclm} \phi(hv) &= ((1\tt h)\tt_H 1) \, \phi(v), \qquad h\in
H, v\in V. \intertext{We denote the vector space of all such
$\phi$ by $\Chom(V,W)$. There is a left action of $H$ on
$\Chom(V,W)$ defined by:} \lbb{hactchom} (h\phi)(v) &=
((h\tt1)\tt_H 1)\, \phi(v).
\end{align}
In the special case $V=W$, we let $\Cend V=\Chom(V,V)$.
\end{definition}

For example, let $A$ be an $H$-\psalg\ and $V$ be an $A$-module.
Then for any $a\in A$ the map $m_a\colon V\to (H\tt H)\tt_H V$
defined by $m_a(v) = a*v$ is an $H$-pseudolinear map. Moreover, we
have $hm_a = m_{ha}$ for $h\in H$.

\begin{remark} \label{**} Consider the map $\rho\colon \Chom(V,W) \tt V \to (H\tt
H)\tt_H W$ given by $\rho(\phi\tt v) = \phi(v)$. By definition it
is $H$-bilinear,
therefore it is a polylinear map 
in $\M^*(H)$. Sometimes, we will  use the notation $\phi * v :=
\phi(v)$ and consider this as a pseudoproduct or {\it
pseudoaction}.
\end{remark}

The associated $x$-products are called {\em Fourier
coefficients\/} of $\phi$ and can be written by a formula
analogous to \eqref{xbracket}:
\begin{equation}\lbb{fourcoe}
\phi_x v = \dsum_i\, \< S(x), f_i {g_i}_{(-1)} \> \, {g_i}_{(2)}
w_i, \qquad \text{if}\quad \phi(v) = \dsum_i\, (f_i\tt g_i) \tt_H
w_i.
\end{equation}
\begin{remark}\label{www}
 Observe that they satisfy a locality relation and an
$H$-sesqui-linearity relation similar to \eqref{codimx} and
\eqref{bl2}:
\begin{equation}\lbb{codimphix}
\codim\{ x\in X \st \phi_x v = 0\} < \infty \quad\text{for
any}\;\; v\in V,
\end{equation}
\begin{equation}\lbb{bl2phix}
\phi_x (hv) = h_{(2)} ( \phi_{h_{(-1)} x} v ).
\end{equation}
Conversely, any collection of maps $\phi_x\in\Hom(V,W)$, $x\in X$,
satisfying relations \eqref{codimphix}, \eqref{bl2phix} comes from
an $H$-pseudolinear map $\phi\in\Chom(V,W)$. Explicitly
\begin{equation*}\lbb{fourcoe2}
\phi(v) = \dsum_i\, ( S(h_i) \tt1)\tt_H \phi_{x_i} v,
\end{equation*}
where $\{h_i\}$, $\{x_i\}$ are dual bases in $H$ and $X$.
\end{remark}

Given $U,V,W$ three $H$-modules, and assuming  that $U$ is finite
(i.e. finitely generated as an $H$-module), there is a unique
polylinear map (see Lemma 10.1 \cite{BDK}).
\begin{displaymath}
\mu\in\Lin( \{ \Chom(V,W), \Chom(U,V)\}, \Chom(U,W))
\end{displaymath}
in $\M^*(H)$, denoted as $\mu(\phi\tt\psi) = \phi*\psi$, such that
\begin{equation}\lbb{phipsiv}
(\phi*\psi)*u = \phi*(\psi*u)
\end{equation}
in $H^{\tt3}\tt_H W$ for $\phi\in\Chom(V,W)$, $\psi\in\Chom(U,V)$,
$u\in U$. More precisely, $\phi*\psi$ is given in terms of the
$x$-products $\phi_x\psi$ by the following formulas
\begin{equation}\lbb{xpcend} (\xp \phi x \psi)_y(v) =
\phi_{x_{(2)}}(\psi_{y x_{(-1)}}(v))=\dsum_i
\phi_{x_{i}}(\psi_{y(h_iS(x))}(v)).
\end{equation}




In the special case $U=V=W$ (finite), we obtain a pseudoproduct
$\mu$ on $\Cend V$, and an action $\rho$ of $\Cend V$ on $V$. More
precisely, for any finite $H$-module $V$, the above pseudoproduct
provides $\Cend V$ with the structure of an \as\ $H$-\psalg\  and
$V$ has a natural structure of a $\Cend V$-module given by $\phi*v
\equiv\phi(v)$.

Moreover, for an \as\ $H$-\psalg\ $A$, giving a structure of an
$A$-module on $V$ is equivalent to giving a homomorphism of \as\
$H$-\psalgs\ from $A$ to $\Cend V$.


 Let $\gc V$ be the Lie
$H$-\psalg\ obtained from the \as\ one $\Cend V$ by the
construction given by (\ref{eq:commutator}). Then $V$ is a $\gc
V$-module. In general, for a Lie $H$-\psalg\ $L$,  giving a
structure of an $L$-module on a finite $H$-module $V$ is
equivalent to giving a homomorphism of Lie $H$-\psalgs\ from $L$
to $\gc V$.

If $V$ is a free $H$-module of finite rank, one can give an
explicit description  of $\gc V$, as follows.

\begin{proposition}(Proposition 10.3 \cite{BDK})\lbb{pcendexpl}
Suppose that  $V=H\tt V_0$, where $H$ acts trivially on $V_0$ and
$\dim V_0 < \infty$. Then $\gc V$ is isomorphic to $H\tt H\tt\End
V_0$, where $H$ acts by  left multiplication on the first factor,
and  the  pseudobracket in $\gc V$ is given by{\rm:}
\begin{equation}\lbb{brackgc}
\begin{split}
[(f\tt a\tt A)*(g\tt b\tt B)] &= (f \tt g a_{(1)}) \tt_H (1 \tt b
a_{(2)} \tt AB)
\\
&- (f b_{(1)} \tt g) \tt_H (1 \tt a b_{(2)} \tt BA).
\end{split}
\end{equation}
\end{proposition}


When $V=H\tt\Kset^n$, we will denote $\gc V$ by $\gc_n$.

\begin{remark}\lbb{ractchom}
Given a   Lie $H$-\psalg\  $L$, and $U,V$  finite $L$-modules, the
formula ($a\in L$, $u\in U$, $\phi\in\Chom(U,V)$)
\begin{equation}\lbb{actchom}
(a*\phi)(u) = a*(\phi*u) - ((\si\tt\id)\tt_H\id) \, \phi*(a*u)
\end{equation}
provides $\Chom(U,V)$ with the structure of an $L$-module. In
particular when $V$ is the base field $\mathbf{k}$, we have the
following
\end{remark}
\begin{definition} Given a  finite module $M$ over a Lie pseudoalgebra
$L$, define de {\it (pseudo-)dual module} of $M$ as
\begin{equation}
              M^*=\Chom (M,\mathbf{k}),
\end{equation}
where $\mathbf{k}$ is a trivial $L$-module with $h\cdot
1=\epsilon(h) 1$ for all $h\in H$.
\end{definition}



 \vskip 1cm


\section{ Duality and Pseudo-bialgebras}\label{sec:pseudobialg}

\vskip .5cm

In this section we extend some results for Lie conformal
 algebras obtained in \cite{BKL} and \cite{L}.
  Let $( L ,[\ \ *\ \ ])$ be a Lie 
 \psalg, and $M$ and $N$ be $L$-modules. We  endow  the ordinary
tensor product of the underlying $H$-modules with an $L$-module
structure. Recall that the action of $H$ in $M\otimes N$ is given
by the coproduct,
 namely, if $h\in H$, $m\in M$ and $ n\in N$:
$$
h\cdot (m\otimes n)\,=\, \Delta(h)( m\otimes n) =h_{(1)} m\otimes
h_{(2)} n.
$$

\begin{lemma}  The $H$-module $M\otimes N$ is an
$L$-module with the following action: ($a\in L$, $m\in M$ and
$n\in N$)

\begin{equation}\label{eq:tensor}
a*(m\otimes n)=\sum_i (f_i {g_{i}}_{(-1)}\otimes 1)\otimes_H (
{g_{i}}_{(2)}m_i\otimes n)+ \sum_j (h_j {l_{j}}_{(-1)}\otimes
1)\otimes_H (m\otimes {l_{j}}_{(2)}n_j),
\end{equation}
if
\begin{equation}\label{aaam}
a*m= \sum_i (f_i\otimes g_i)\otimes_H m_i= (f_i
{g_{i}}_{(-1)}\otimes 1)\otimes_H {g_{i}}_{(2)}m_i \in\, (H\otimes
H)\otimes_H M ,
\end{equation}
and
$$a*n= \sum_j (h_j\otimes l_j)\otimes_H n_j=\sum_j (h_j {l_{j}}_{(-1)}\otimes
1)\otimes_H  {l_{j}}_{(2)}n_j\in\, (H\otimes H)\otimes_H N.
$$

\end{lemma}

\begin{proof} Observe that Lemma \ref{lhhh} and the arguments that
produced (\ref{psbr2}) show that the action $a*m$ ($a\in L, m\in
M$) in (\ref{aaam}) can be written
\begin{equation}\label{a*m}
a*m= \sum_i (f_i {g_i}_{(-1)}\otimes 1)\otimes_H  {g_i}_{(2)}m_k
\in\, (H\otimes H)\otimes_H M ,
\end{equation}
hence, the action $a*m$ ($a\in L, m\in M$) can be written uniquely
in the form $\dsum_i (h_i\otimes 1)\otimes_H c_i$, where $\{h_i\}$
is a fixed $\mathbf{k}$-basis of $H$. In this way,
\begin{equation}\label{eq:tensorh}
a*(m\otimes n)=\sum_k (h_k\otimes 1)\otimes_H ( m_k\otimes n)+
\sum_l (h'_l\otimes 1)\otimes_H (m\otimes n_l),
\end{equation}
if
\begin{equation*}
a*m= \sum_k (h_k\otimes 1)\otimes_H  m_k \in\, (H\otimes
H)\otimes_H M ,
\end{equation*}
and
$$a*n= \sum_l (h'_l\otimes 1)\otimes_H n_l\in\, (H\otimes
H)\otimes_H N.
$$
First of all we have to show the H-bilinearity of the action
defined in (\ref{eq:tensor}).  Observe that, in general, using
(\ref{sigfour}) and the $H$-linearity of $M$, we have ($f,g\in H,
a\in L$ and $m\in M$)
\begin{align*}
f a*g  m \ &=  ((f\otimes g)\otimes_H 1)(a*m)=
\sum_k (fh_k\otimes g)\otimes_H  m_k\nonumber\\
\,&=  \sum_k (fh_kg_{(-1)}\otimes 1)\otimes_H  g_{(2)}m_k.
\end{align*}
Therefore,  using the cocommutativity of $H$, (\ref{sigfour}) and
(\ref{eq:tensorh}),
\begin{align}
fa*g(m\otimes n)\,  &=  fa*(g_{(1)}m\otimes g_{(2)}n)\nonumber\\
\,&= \sum_k (fh_kg_{(-1)}\otimes 1)\otimes_H
(g_{(2)}m_k \otimes g_{(3)}n)\nonumber\\
\, &\, \qquad \qquad\qquad + \sum_ l (fh'_lg_{(-2)}\otimes
1)\otimes_H (g_{(1)}m
\otimes g_{(3)}n_l),\nonumber\\
\,&=\sum_k (fh_kg_{(-1)}\otimes 1)\Delta(g_{(2)})\otimes_H
(m_k \otimes n)\nonumber\\
\, &\, \qquad \qquad \qquad + \sum_ l (fh'_lg_{(-1)}\otimes
1)\Delta(g_{(2)})\otimes_H (m \otimes n_l),\nonumber\\
\,&= \sum_k (fh_k\otimes g)\otimes_H ( m_k\otimes n)+ \sum_l
(fh'_l\otimes g)\otimes_H (m\otimes n_l)\nonumber\\
\,&= ((f\otimes g)\otimes_H 1) (a*(m\otimes n)),
\end{align}
proving the $H$-bilinearity.

To prove that $M\otimes N$ is an $L$-module we will introduce the
following notation that simplify (\ref{a*m}): for $a\in L$ and
$m\in M$, we denote
\begin{equation}\label{notation}
a*m=\sum_{(a,m)}(h^{a,m}\otimes 1)\otimes_H m_a=(h^{a,m}\otimes
1)\otimes_H m_a.
\end{equation}
where we avoided the sum that is implicitly understood. With this
notation, we have that (\ref{eq:tensorh}) can be rewritten as
 \begin{eqnarray}\label{eq:tensnew}
 a*(m\otimes n)= (h^{a,m}\otimes 1)\otimes_H(m_a\otimes n)+
 (h^{a,n}\otimes 1)\otimes_H(m\otimes
 n_a).
 \end{eqnarray}
Thus, by the composition rule
\begin{align} \label{eq:mod1}
a*(b*(m\otimes n))\, &=  (h^{a,m_b}\otimes h^{b,m}\otimes
1)\otimes_H((m_b)_a\otimes n)\nonumber\\
&\ \ \ + ( h^{a,n}\otimes h^{b,m}\otimes
1)\otimes_H(m_b\otimes n_a)\\
&\ \ \ +  (h^{a,m}\otimes h^{b,n}\otimes
1)\otimes_H(m_a\otimes n_b)\nonumber\\
&\ \ \ +  (h^{a,n_b}\otimes h^{b,n}\otimes 1)\otimes_H(m\otimes
(n_b)_a).\nonumber
\end{align}
Similarly, interchanging the roles of $a$ and $b$ in
(\ref{eq:mod1}),
\begin{align*}
((\si\tt\id)\tt_H\id)(b*(a*(m\otimes n))\ &=  (h^{a,m}\otimes
h^{b,m_a}\otimes
1)\otimes_H((m_a)_b\otimes n)\nonumber\\
&\ \ \ +  (h^{a,m}\otimes h^{b,n}\otimes
1)\otimes_H(m_a\otimes n_b)\nonumber\\
&\ \ \ +  (h^{a,n}\otimes h^{b,m}\otimes
1)\otimes_H(m_b\otimes n_a)\nonumber\\
&\ \ \ +  (h^{a,n}\otimes h^{b,n_a}\otimes 1)\otimes_H(m\otimes
(n_a)_b).
\end{align*}

\noindent Then,
\begin{align}\label{eq:modlhs}
 & a*(b*(m\otimes n)) - ((\si\tt\id)\tt_H\id)(b*(a*(m\otimes n))=\\
 &\, \qquad (h^{a,m_b}\otimes h^{b,m}\otimes
1)\otimes_H((m_b)_a\otimes n) - (h^{a,m}\otimes h^{b,m_a}\otimes
1)\otimes_H((m_a)_b\otimes n)\nonumber\\
&\,\qquad  + (h^{a,n_b}\otimes h^{b,n}\otimes 1)\otimes_H(m\otimes
(n_b)_a) - (h^{a,n}\otimes h^{b,n_a}\otimes 1)\otimes_H(m\otimes
(n_a)_b).\nonumber
\end{align}
On the other hand, if $[a*b]=\sum (h_i\otimes 1)\otimes_He_i$, by
the composition rules (\ref{abc*3'}),
\begin{align} \label{eq:modrhs}
[a*b]*(m\otimes n)= & \ (h_i(h^{e_i,m})_{(1)}\otimes
(h^{e_i,m})_{(2)}\otimes 1)\tt_H (m_{e_i}\otimes n)\nonumber\\
&+  (h_i(h^{e_i,n})_{(1)}\otimes (h^{e_i,n})_{(2)}\otimes 1)\tt_H
(m\otimes n _{e_i}).
\end{align}
 Now, using the fact that  $M$ and $N$ are $L$ modules themselves, it is
 immediate to check that the first summand in (\ref{eq:modrhs}) corresponds
 exactly with the first two in  (\ref{eq:modlhs}), and similarly with
 the remaining ones, finishing the proof.
\end{proof}

Let $M$ be an $L$-module. Using  the arguments that produced
(\ref{psbr2}) and Lemma \ref{lhhh}, observe that if $f\in M^*$ and
$m\in M$, then $f(m)$ can be uniquely written as $(g_{f,m}\tt
1)\tt_H 1$. We have the following useful result (cf. Proposition
6.1 in \cite{BKL}).

\begin{proposition}\label{prop:ddd} Let $M$ and $N$ be two $L$-modules.
Suppose that $M$ has  finite rank as an $H$-module.
 Then $M^{*}\otimes N\simeq $ Chom$(M,N)$ as $L$-modules, where
 the correspondence
 $\phi: M^{*}\otimes N\longrightarrow\Chom(M,N)$ is  given by,
 \begin{equation}
 \big[\phi(f\otimes n)\big] (m)=(1\otimes S(g_{f,m}))\otimes_H n,
 \end{equation}
 if $f\in M^*$, $ m\in M$, $n\in
 N$ and $f(m)=(g_{f,m}\otimes 1)\otimes_H1\in (H\otimes H)\otimes_H\mathbf{k}$.
\end{proposition}

\begin{proof} Let us check that $\phi(f\tt v)\in\Chom(M,N)$. Since $f\in
M^*=\Chom(M,\mathbf{k})$, using (\ref{sigfour}) combined with
(\ref{cou}), and recalling that the $H$-module structure of
$\mathbf{k}$ is given via the counit, we have that $f(hm)=((1\tt
h)\tt_H 1)(f(m))= (g_{f,m}\tt h)\tt_H 1= (g_{f,m}h_{(-1)}\tt
1)\tt_H\varepsilon(h_{(2)})= (g_{f,m}S(h)\tt 1)\tt_H 1$, for all
$m\in M$ and $h\in H$. Thus, by  (\ref{sabsbsa}),
\begin{align}
\big[\phi(f\tt n)\big](hm)&= (1\tt S( g_{f,m}S(h)))\tt_H
 n= (1\tt hS( g_{f,m}))\tt_H
 n\nonumber\\
&= ((1\tt h)\tt_H 1)((1\tt S( g_{f,m}))\tt_H
n)\nonumber\\
&= ((1\tt h)\tt_H 1)\big[\phi(f\tt n)\big](m).\nonumber
\end{align}
Now, we will check that this identification given via $\phi$ is
$H$-linear. By the $H$-module structure of the tensor product of
modules and (\ref{cou2}), we have that,
\begin{eqnarray}
\big[\phi(h(f\tt n))\big](m)&=&\big[\phi( h_{(1)}f\tt
h_{(2)}n)\big](m)=(1\tt
S(h_{(1)} g_{f,m}))\tt_H h_{(2)}n\nonumber\\
&=&[(1\tt S(g_{f,m}))(1\tt S(h_{(1)}))\Delta(h_{(2)} )]\tt_H
n\nonumber\\
&=&(h\tt S(g_{f,m}))\tt_H n=((h\tt 1)\tt_H 1)\big[\phi(f\tt
n)\big](m)
\nonumber\\
&=&\big[h\ \phi(f\tt n)\big](m).\nonumber
\end{eqnarray}
Now, we will show that $\phi$ is a morphism of $L$-modules. To
keep simple expressions, we shall use the notation introduced in
(\ref{notation}) and Remark \ref{**}. Therefore, we consider
$\phi(f\tt n)*m:=[\phi(f\tt n)](m)$ and since we have already
shown the  $H$-bilinearity, it is actually a polylinear map,
therefore using the composition rule (\ref{abc*3'}) and the action
(\ref{eq:tensnew}), we have

\begin{equation}\label{I-II}
\phi(a*(f\tt n))*(m)=(h^{a,f}\tt 1\tt S(g_{f_a,m})\tt_H n+
(h^{a,n}\tt 1\tt S(g_{f,m}))\tt_H  n_a,
\end{equation}

\vskip .3cm \noindent if $f_a*(m)=(g_{f_a,m}\tt 1)\tt_H 1$ and $
a*(f\tt n)=(h^{a,f}\tt 1)\tt_H (f_a\tt n)+ (h^{a,n}\tt 1)\tt_H
(f\tt n_a)$.

On the other hand, by (\ref{actchom}), (\ref{abc*3'}) and
(\ref{abc*6'}),
\begin{align}\label{III-IV}
 \left[a*\phi  (f\tt n)\right]& *(m)= \  a*((1\tt S(g_{f,m}))\tt_H
 n)\qquad \qquad
\nonumber\\
\,&\qquad\qquad-((\si\tt\id)\tt_H\id)[\phi(f\otimes
n)]*((h^{a,m}\otimes
1)\tt_H m_a)\nonumber\\
= \ & (h^{a,n}\tt 1 \tt S(g_{f,m}))\tt_H n_a\nonumber\\
\,&\qquad- ((\si\tt\id)\tt_H\id)((1\tt
h^{a,m}(g_{f,m_a})_{(-1)}\tt (g_{f,m_a})_{(-2)})\tt_H n) \nonumber\\
= \ & (h^{a,n}\tt 1 \tt S(g_{f,m}))\tt_H n_a\nonumber\\
\,&\qquad-  (h^{a,m}(g_{f,m_a})_{(-1)}\tt1\tt
(g_{f,m_a})_{(-2)})\tt_H n.
\end{align}
Now, comparing (\ref{I-II}) and (\ref{III-IV}), it is enough to
show that
\begin{equation}\label{P}
-(h^{a,f}\tt 1\tt S(g_{f_a,m}))\tt_H n=
(h^{a,m}(g_{f,m_a})_{(-1)}\tt1\tt (g_{f,m_a})_{(-2)})\tt_H n.
\end{equation}
But this follows from two different ways to compute $(a*f)(m)$. By
(\ref{actchom}) and (\ref{identity}), we have that
\begin{align}
(a*f)(m)&=-((\sigma\tt\id)\tt_H\id)(f*((h^{a,m}\tt 1)\tt_H m_a))\nonumber\\
&= -((\sigma\tt\id)\tt_H\id) ((g_{f,m_a}\tt h^{a,m}\tt 1)\tt_H 1
)\nonumber\\
&=- (h^{a,m}\tt g_{f,m_a} \tt 1)\tt_H 1
\nonumber\\
&=- (h^{a,m} (g_{f,m_a})_{(-1)} \tt 1\tt (g_{f,m_a})_{(-2)})\tt_H
1
\end{align}
and by notation (\ref{notation}) and (\ref{identity}),
\begin{align*}
(a*f)(m)= \ & (h^{a,f}(g_{f_a,m})_{(1)}\tt (g_{f_a,m})_{(2)}\tt
1)\tt_H 1 \\
=\ & (h^{a,f}\tt 1\tt S(g_{f_a,m}))\tt_H 1,
\end{align*}
hence (\ref{P}) follows. Now we have to prove the injectivity.
Suppose that for all $m\in M$ we have $ 0=\phi(\sum_{i} f_i\tt
n_i)(m)=\sum_i (1\tt S(g_{f_i,m}))\tt_H n_i.$ Then $g_{f_i,m}=0$
for all $m\in M$, since $S$ is bijective. Therefore $f_i=0$ for
all $i$.

It remains to prove that $\phi$ is surjective. Let $g\in
$Chom$(M,N)$ and $M=\oplus_{i=1}^n H m_i$. Then there exist
$h_{ij}\in H$ and $n_{ij}\in N$ such that $g(m_i)=\sum_j (1\tt
h_{ij})\tt_H n_{ij}$. Now, we define $f_{ij}\in M^*$ by
$f_{ij}(m_k)=[(S(h_{ij})\tt 1)\tt_H 1]\delta_{ik}$. Then
$g=\phi(\sum_{i,j} f_{ij}\tt n_{ij})$ since $\phi(\sum_{i,j}
f_{ij}\tt n_{ij})(m_k)=\sum_{i,j} \delta_{ik}(1\tt h_{ij})\tt_H
n_{ij}=\sum_{j} (1\tt h_{kj})\tt_H n_{kj}=g(m_k)$, finishing our
proof.
\end{proof}


As a motivation  for the definition of $H$-coalgebra and
pseudo-bialgebra, we used the  cohomology theory of pseudoalgebras
developed in \cite{BDK}, in order to get to  the right notion of
cocycle that will be the compatibility condition between
pseudobracket and coproduct. See Section \ref{subcohom}, for a
brief review of the basics of this theory.

We have the following definition:

\begin{definition} A {\it Lie  H-coalgebra} $R$ is an $H$-module,
endowed with  an $H$-homomorphism
\begin{displaymath}
\delta:R\to \wedge^2 R
\end{displaymath}
such that
\begin{equation}\label{co-jacobi-delta}
(I\otimes \delta) \delta - \tau_{12}(I\otimes \delta) \delta=
(\delta\otimes I) \delta.
\end{equation}
where $\tau_{12}=(1,2)\in \mathcal{S}_3$.
\end{definition}
This is nothing but the standard definition of a Lie coalgebra,
compatible  with the  $H$-module structure of $R$.

\vskip .4cm

In this section we will give the answer to the following natural
question: Does the "dual" of one structure produce the other, at
least in finite rank?  The answer is given by the Theorem
\ref{duality-thm}, below. But first we will need the following
definition.

\begin{definition} Let $L$ be a finite free $H$-module with
basis $\{a_i\}_{i=1}^n$. The {\it dual basis} of $\{a_i\}_{i=1}^n$
in $L^*$ is defined by the set  $\{a^j\}_{j=1}^n$, where each
$a^j\in L^*= \hbox{Chom}(L,\kk)$ is given by
 \begin{eqnarray} \label{dualbasis}
 a^i*(a_j)=(1\otimes1)\otimes_H\delta_{ij}.
 \end{eqnarray}
\end{definition}
It is easily  checked that, with this definition $\{a^j\}_{j=1}^n$
is a linearly independent set  such that $H$-generates $L^*$.

 \begin{theorem} \label{duality-thm}
(a)  Consider $L=\oplus_{i=1}^N H a_i$ a finite free Lie
$H$-pseudoalgebra,  with pseudobracket given by
$$
[a_i*a_j]=\sum_{k=1}^N (h_k^{ij}\otimes l_k^{ij})\otimes_H a_k.
$$
Let $L^*=\hbox{\rm Chom}(L,\kk)= \oplus_{i=1}^N H a^i$ the dual of
$L$ where $\{a^i\}$ is the dual basis corresponding to $\{a_i\}$.
Define $\delta:L^*\longrightarrow L^*\otimes L^*$ as follows:
\begin{eqnarray}
 \label{delta}
\delta(a^k)=\sum_{i,j} S(h_k^{ij})a^i\otimes S(l_k^{ij})a^j
\end{eqnarray}
and extend it $H$-linearly, i.e.
$\delta(ha^k)=\Delta(h)\delta(a^k)$. Then $(L^*,\delta)$ is a Lie
$H$-coalgebra.

\

(b) Conversely, let $(R,\delta)$ be a finite Lie $H$-coalgebra.
Then  the left $H$-module $R^*=\hbox{\rm Chom}(R,\mathbf{k})$ is a
Lie $H$-conformal algebra  with  the $x$-brackets defined by
\begin{equation} \label{x-prod-psudo[]}
[f\,_x\,g\,]_y(r)=\sum
f_{x_{(2)}}(r_{(1)})\,g_{yx_{(-1)}}(r_{(2)})
\end{equation}
with $f,\,g\in R^*$, $r\in R$ and $x,\,y\in X=H^*$, where
$\delta(r)=\sum r_{(1)}\otimes r_{(2)}$.
\end{theorem}

\begin{proof}  (a)  Due to the
skew-commutativity of the pseudobracket of $L$, $[a_i*a_j]=\sum_k
(h^{ij}_k\tt l^{ij}_k)\tt_H a_k= -\sum_k (l^{ji}_k\tt
h^{ji}_k)\tt_H a_k=-(\sigma\tt_H\id)[a_j*a_i]$. Then we have that
\begin{align*}
\delta(a^k)&=\sum_{i,j} S(h^{ij}_k)a^i\tt
S(l^{ij}_k)a^j\nonumber\\
&=- \sum_{i,j} S(l^{ji}_k)a^i\tt
S(h^{ji}_k)a^j\nonumber\\
&=-\sigma(\delta(a^k)),
\end{align*}
showing that $\delta(a^k)\in \wedge^2(L)$. Now we have to check
the co-Jacobi condition for $\delta$. Using the notation
$[a_i*a_j]=\sum_k (h^{ij}_k\tt l^{ij}_k)\tt_H a_k$ and the
composition rules (\ref{abc*3'}) and (\ref{abc*6'}), together with
the Jacobi identity in $L$,
\begin{align} \label{jacobi/co}
0&=[a_i*[a_j*a_l]]-(\sigma\tt\id)[a_j*[a_i*a_l]]-[[a_i*a_j]*a_l]\nonumber\\
&=\sum_{k,s} (h^{ik}_s\tt h^{jl}_k(l^{ik}_s)_{(1)}\tt
l^{jl}_k(l_s^{ik})_{(2)})\tt_H a_s
\nonumber\\
&\qquad \quad - \sum_{k,s} (h^{il}_k(l^{jk}_s)_{(1)}\tt
h^{jk}_s\tt
l^{il}_k(l_s^{jk})_{(2)})\tt_H a_s\nonumber\\
&\qquad\qquad\qquad - \sum_{k,s} (h^{ij}_k(h^{kl}_s)_{(1)}\tt
l^{ij}_k(h_s^{kl})_{(2)}\tt l^{kl}_s)\tt_H a_s.
\end{align}
Now,
\begin{align*}
(\id\tt\delta)(\delta (a^s))&= (\id\tt\delta)\left(\sum_{i,k}
S(h^{ik}_s)a^i\tt S(l^{ik}_s)a^k\right)\nonumber\\
&=\sum_{i,k,j,l} S(h^{ik}_s)a^i \tt \Delta(S(l^{ik}_s))\left(
S(h^{jl}_k)a^j\tt S(l^{jl}_k)a^l\right)\nonumber\\
&= \sum_{i,j,k,l} \left((S\tt S\tt S)(h^{ik}_s\tt
h^{jl}_k(l^{ik}_s)_{(1)}\tt
l^{jl}_k(l_s^{ik})_{(2)})\right)(a^i\tt a^j\tt a^l).
\end{align*}
Similarly,
\begin{eqnarray*}
-\tau_{12}(\id\tt\delta)(\delta(a^s))=-\sum_{i,j,k,l} \left((S\tt
S\tt S)(h^{il}_k(l^{jk}_s)_{(1)}\tt h^{jk}_s\tt
l^{il}_k(l_s^{jk})_{(2)})\right)(a^i\tt a^j\tt a^l),
\end{eqnarray*}
and
\begin{eqnarray*}
(\delta\tt\id)(\delta(a^s))= \sum_{i,j,k,l} \left((S\tt S\tt
S)(h^{ij}_k(h^{kl}_s)_{(1)}\tt l^{ij}_k(h_s^{kl})_{(2)}\tt
l^{kl}_s)\right)(a^i\tt a^j\tt a^l).
\end{eqnarray*}
Therefore comparing with (\ref{jacobi/co}), we have that $\delta$
satisfies the co-Jacobi condition, namely,
\begin{eqnarray*}
(\id\tt\delta)(\delta (a^s))
-\tau_{12}(\id\tt\delta)(\delta(a^s))-(\delta\tt\id)(\delta(a^s))=0.
\end{eqnarray*}

\noindent (b) We define our candidate for the pseudobracket in
$R^*$ in terms of its Fourier coefficients (cf.  formula (9.21) in
\cite{BDK} or (\ref{xpcend})):
\begin{displaymath}
[f_x g]_y (r) = \xp f {x_{(2)}}(r_{(1)})\, {\xp g{y x_{(-1)}}
(r_{(2)})} = \tsum_i\, \xp f {x_i}(r_{(1)}) \,{ \xp g {y (h_i
S(x))} (r_{(2)})},
\end{displaymath}
where $\{h_i\}$ and $\{x_i\}$ are dual basis in $H$ and $X$
respectively. The $H$-sesqui-linearity properties of $[f_x g]_y
(r)$ with respect to $x$ and $y$ are tedious but straightforward.

By properties \eqref{fils1}, \eqref{fils3}, \eqref{fils2} of the
filtration $\{\fil_n X\}$, if $y\in F_nX$, then $yx_{(-1)}\in
F_nX$ for all $x\in X$. Thus by locality of $g$, it follows that
for each fixed $x\in X$, $r\in R$ there is an $n$ such that $[f_x
g]_y (r) = 0$ for $y\in \fil_n X$. Therefore, by Remark \ref{www},
for each $x\in X$ we have that $ [f _x g] \in
\Chom(R,\mathbf{k})$.

In order to see that $[f*g]$  is well defined, we need to check
that $[f _x g]$ satisfies locality, i.e. for each $f,g\in R^*$
there exists $ n\in\mathbb{N}$ such that $[f_xg]=0$ for all $x\in
\fil_nX$. By the locality of $f$ and $g$, for each term of
$\delta(r)=\sum r_{(1)}\tt r_{(2)}$, there are  $n_1$ and $n_2$
such that $f_{x_{(2)}}(r_{(1)})=0$ if $x_{(2)}\in \fil_{n_1} X$,
and  $ g_{yx_{(-1)}}(r_{(2)}) = 0$ if $yx_{(-1)}\in \fil_{n_2}X$.
Thus taking $n$ big enough and using  \eqref{fils2}  we have that
$x_{(2)}$ or $x_{(-1)}$ belongs to $\fil_n(X)$. Since we have that
$y\,\fil_nX\subseteq \fil_nX$ for all $y\in X$ then, we conclude
that for each $r\in R$ there exists $n$ such that $[f_xg]_y(r)=0$
for all $y\in X$ and for all $x\in \fil_n X$.

 Since $R$ is
finite, we can choose an $n$ that works for all $r$ belonging to a
set of generators of $R$ over $H$. Now the $H$-sesqui-linearity of
$[f_x g]_y (r)$ with respect to $y$ (for fixed $x$) implies that
$[f_x g]_y (r) = 0$ for all $y$ and $r$. Hence $[f_ x g] = 0$ for
$x\in \fil_n X$.

To finish our proof, we need to check the skew-commutativity
(\ref{sksy}) and the Jacobi identity (\ref{jid}) for
(\ref{x-prod-psudo[]}). In order to see the skew-commutativity we
need to proof that $[f_x g]=-\sum \langle x,h_{i_{(-1)}}\rangle\,
h_{i_{(-2)}}\,[g_{x_i}f]$. Evaluating the right hand side of this
equation in $r$ and using the skew-symmetry of $\delta$, we have
that

\begin{align*}
 -\sum\langle x,h_{i_{(-1)}}\rangle\,
 h_{i_{(-2)}}\,[g_{x_i}f]_y(r)&=-\sum
\langle x,h_{i_{(-1)}}\rangle\, \,[g_{x_i}f]_{yh_{i_{(-2)}}}(r)\nonumber\\
&=-\sum \langle x,h_{i_{(-1)}}\rangle\,
\,g_{x_{i_{(2)}}}(r_{(1)})\,f_{(yh_{i_{(-2)}})x_{i_{(-1)}}}(r_{(2)})\nonumber\\
&= \sum \langle
x,h_{i_{(-1)}}\rangle\,\,f_{(yh_{i_{(-2)}})x_{i_{(-1)}}}(r_{(1)})\,
g_{x_{i_{(2)}}}(r_{(2)}).
\end{align*}

Since  $[f_x g]_y
(r)=f_{x_{(2)}}(r_{(1)})\,g_{yx_{(-1)}}(r_{(2)})$, in order to
prove the skew-commutativity is enough to show that

\begin{align}\label{eq:skew.lemma}
\sum \langle
x,h_{i_{(-1)}}\rangle\,\left((yh_{i_{(-2)}})x_{i_{(-1)}}\tt
x_{i_{(2)}}\right)=x_{(2)}\tt yx_{(-1)}.
\end{align}
For $k,l\in H$,       we have

\begin{align*}
& \sum \langle x,h_{i_{(-1)}}\rangle
\,\left((yh_{i_{(-2)}})x_{i_{(-1)}}\tt
x_{i_{(2)}}\right)(k\tt l)=\nonumber\\
&=\sum \langle x_{(1)},h_{i_{(-1)}}\rangle \langle
x_{(2)},1\rangle\,\left((yh_{i_{(-2)}})x_{i_{(-1)}}\tt
x_{i_{(2)}}\right)(k\tt l) \qquad\qquad\qquad \hbox{by (\ref{x,fg} )} \nonumber\\
&=\sum \langle x_{(-1)},h_{i_{(1)}}\rangle \langle
x_{(2)},1\rangle\,\langle(yh_{i_{(-2)}})x_{i_{(-1)}},
k\rangle\langle
x_{i_{(2)}},l\rangle\nonumber\\
&= \sum \langle x_{(-1)},h_{i_{(1)}}\rangle \langle
x_{(2)},1\rangle\,\langle(yh_{i_{(-2)}}), k_{(1)}\rangle\langle
x_{i_{(-1)}},k_{(2)}\rangle\langle
x_{i_{(2)}},l\rangle  \qquad\qquad \hbox{by (\ref{xy,f} )} \nonumber\\
&=  \sum \langle x_{(-1)},h_{i_{(2)}}\rangle \langle
x_{(2)},1\rangle\,\langle(k_{(-1)}y),h_{i_{(1)}} \rangle\langle
x_{i_{(1)}},k_{(-2)}\rangle\langle x_{i_{(2)}},l\rangle \nonumber\\
&\qquad \qquad \qquad \qquad \qquad \qquad \hbox{by (\ref{hx}),
(\ref{xh} ),
(\ref{sxhxsh} ) and cocommutativity of $H$}\nonumber\\
&= \sum \langle (k_{(-1)}y)x_{(-1)},h_{i}\rangle \langle
x_{(2)},1\rangle\, \langle x_{i},k_{(-2)}l\rangle
\qquad\qquad \qquad \quad\hbox{by (\ref{xy,f} ) and (\ref{x,fg})}\nonumber\\
&= \langle (k_{(-1)}y)x_{(-1)},k_{(-2)}l\rangle \langle
x_{(2)},1\rangle \qquad\qquad\qquad\qquad\qquad
\qquad\qquad\hbox{by (\ref{AAA} )}.
\end{align*}

\vskip .5cm

\noindent Since $(hy)x_{(-1)}\tt x_{(2)}=h_{(1)}(yx_{(-1)})\tt
h_{(2)}x_{(2)}$, we get

\begin{align*}
\langle (k_{(-1)}y)x_{(-1)},k_{(-2)}l\rangle \langle
x_{(2)},1\rangle&=\langle
k_{(-1)_{(1)}}(yx_{(-1)}),k_{(-2)}l\rangle \langle k_{(-1)_{(2)}}x_{(2)},1\rangle \nonumber\\
&=\langle yx_{(-1)},k_{(1)_{(1)}}k_{(-2)}l\rangle \langle
x_{(2)},k_{(1)_{(2)}}
 \rangle \qquad\,\, \hbox{by (\ref{hx})}\nonumber\\
 &=\langle yx_{(-1)},l\rangle \langle
x_{(2)},k \rangle , \qquad\quad\quad \, \,\,\,\,\hbox{by
(\ref{antip}) and (\ref{cou})}
\end{align*}
proving the identity in (\ref{eq:skew.lemma}).  To finish, we
still have to check Jacobi identity. We have that
\begin{align}\label{aJacobi}
[f_x[g_yl]]_z(r)-
 & [g_y[f_xl]]_z(r)=f_{x_{(2)}}(r_{(1)})\,[g_yl]_{zx_{(-1)}}(r_{(2)})-g_{y_{(2)}}(r_{(1)})\,
[f_xl]_{zy_{(-1)}}(r_{(2)})\nonumber\\
 &=f_{x_{(2)}}(r_{(1)})\,g_{y_{(2)}}(r_{(2)_{(1)}}) \,
l_{(zx_{(-1)})y_{(-1)}}(r_{(2)_{(2)}}) \\
&\qquad\qquad- g_{y_{(2)}}(r_{(1)})\,
f_{x_{(2)}}(r_{(2)_{(1)}})\,l_{(zy_{(-1)})x_{(-1)}}(r_{(2)_{(2)}}),\nonumber
\end{align}
and
\begin{align}\label{bJacobi}
[[f_{x_{(2)}}g]_{yx_{(1)}}l]_z(r)&=
[f_{x_{(2)}}g]_{(yx_{(1)})_{(2)}}(r_{(1)})\,
l_{z(yx_{(1)})_{(-1)}}(r_{(2)}) \\
&= f_{x_{(2)_{(2)}}}(r_{(1)_{(1)}}) \,
g_{(yx_{(1)})_{(2)}x_{(2)_{(-1)}}}
(r_{(1)_{(2)}})\,l_{z(yx_{(1)})_{(-1)}} (r_{(2)}).\nonumber
\end{align}
Due to the co-Jacobi condition of $\delta$ given in
(\ref{co-jacobi-delta}), comparing (\ref{aJacobi}) and
(\ref{bJacobi}), it is enough to show that
\begin{align}\label{nose}
x_{(2)_{(2)}}\tt(yx_{(1)})_{(2)}x_{(2)_{(-1)}}\tt
z(yx_{(1)})_{(-1)}&= x_{(2)}\tt y_{(2)}\tt
(zx_{(-1)})y_{(-1)}\nonumber\\
&= x_{(2)}\tt y_{(2)}\tt (zy_{(-1)})x_{(-1)}.
\end{align}
Due to the  commutativity and associativity of $X$, the last
equality is immediate. Since $\Delta(xy)=x_{(1)}y_{(1)}\tt
x_{(2)}y_{(2)}$, we have
\begin{align*}
x_{(2)_{(2)}}\tt & (yx_{(1)})_{(2)}x_{(2)_{(-1)}}\tt
z(yx_{(1)})_{(-1)}=\nonumber\\
&= (1\tt y_{(2)}\tt zy_{(-1)})(x_{(2)_{(2)}}\tt
x_{(1)_{(2)}}x_{(2)_{(-1)}}\tt
x_{(1)_{(-1)}})\nonumber\\
&= x_{(2)}\tt y_{(2)}\tt (zx_{(-1)})y_{(-1)},
\end{align*}
proving the first equality of (\ref{nose}) and finishing our
proof.
\end{proof}

Motivated by the definition of the differential of a 1-cochain in
the reduced complex of a Lie $H$-pseudoalgebra (Sect 15.1
\cite{BDK}), we introduce the following notion.

\begin{definition}A Lie $H$-{\it pseudo-bialgebra} is a triple $(L,[\ *\,],\delta)$ such that $(L,[\ *\,])$ is a pseudoalgebra,
$(L,\delta)$ is an $H$-coalgebra and they satisfy the cocycle
condition:
$$
a*\delta(b) - (\sigma\otimes_H 1) b*\delta(a)=\delta([a*b])
$$
for all $a$ and $b$ in $L$.
\end{definition}

\

\begin{example} Let $(\fg, [\  , \ ], \bar\delta)$ be a Lie bialgebra.
Now, it is easy to check that the pseudoalgebra
Cur\,$\fg=H\otimes \fg$ has a natural  Lie pseudo-bialgebra
structure given by:
\begin{displaymath}
\delta(f\otimes a)= f\cdot \overline{\delta}(a),
\end{displaymath}
for $f\otimes a\in$ Cur\,$\frak g$. But not all the bialgebra
structures on Cur$(\fg)$ are of this form, as it is shown in  the
next example.
\end{example}

\begin{example}\label{currr} Consider the rank 2 solvable Lie
pseudoalgebra
\begin{equation*}
    L_p =Ha\oplus Hb,
\end{equation*}
with $*$-bracket (extended by skew-symmetry
 and sesquilinearity and) given by
\begin{displaymath}
[a * a]=0=[b * b], \quad [a * b]= (p\otimes1) \otimes_H b,
\end{displaymath}
where $p\in H$. We shall not consider the most general case where
$p\otimes 1$ is replaced by $\alpha\in H\otimes H$. We do not plan
to give an exhaustive classification of  Lie pseudo-bialgebra
structures on $L_p$, instead, we shall study pseudo-bialgebra
structures on $L_p$ whose underlying coalgebra structure comes
from the dual of a solvable Lie pseudoalgebra $L_h$, with $h\in
H$. That is, fix $h\in H$, then by applying Theorem
\ref{duality-thm} to $L_h$ we obtain a Lie $H$-coalgebra structure
on $L_h$ by taking $\delta_h: L_h\to \wedge^2 L_h$ given by
\begin{equation*}
\delta_h(a)=0, \quad \delta_h(b)=  S(h) a\otimes  b -  b\otimes
S(h)a.
\end{equation*}
By a simple computation it is possible  to  show that $\delta_h$
is a Lie pseudo-bialgebra structure on $L_p$ if and only if
\begin{equation}\label{zzzzz}
    (S\otimes 1) \Delta(S(h)p)=-(1\otimes S)\Delta(S(h)p)
\end{equation}

In the special case of $p= 1$, we have that
$L_p\simeq\,$Cur$(T_2)$ where $T_2$ is the 2-dimensional Lie
algebra considered in Examples 2.2 and 3.2 in \cite{ES}. In this
case every  $h$ satisfying (\ref{zzzzz}) produce a non-isomorphic
Lie pseudo-bialgebra structure in Cur$(T_2)$, obtaining
pseudo-bialgebra structures that do not come from bialgebra
structures in $T_2$ as in the previous example. Moreover, in order
to see how different is the situation from the classical case,
observe that if $h$ satisfies (\ref{zzzzz}) and $S(h)=-h$  (which
is the case if $h\in \fg\subset \mathcal{U}(\fg )=H$), then
$\delta_h= d(1\otimes_H r) $, where $r= \frac 1 2( a\otimes ha-
ha\otimes a)$ (cf. (\ref{diga0}) and  (\ref{eq:deltar}) below),
showing that there are coboundary structures $\delta =dr$ (see
next section for the definition) in Cur$(T_2)$ with $\delta(a)=0$
and such structures are not present in the Lie algebra $T_2$ (see
Example 3.2 in \cite{ES}).
\end{example}

\begin{example} Recall that $gc_1$ can be identified with
$H\otimes H$   with pseudobracket defined as follows (see
(\ref{brackgc})):
$$
[(f\otimes a) * (g\otimes b)]= (f\otimes g a_{(1)})\otimes_H
(1\otimes b a_{(2)})-(fb_{(1)}\tt g)\tt_H(1\tt ab_{(2)}),
$$
for $f\otimes a$ and   $g\otimes b$ in $H\otimes H$. By
straightforward computations, it is possible to show that given
$r=(f\tt 1)\wedge (g\tt 1)\in gc_1\wedge gc_1$,

\begin{align*}
 \delta_r(1\otimes a)&=(fa_{(1)})_{(-1)}\cdot\big((fa_{(1)})_{(2)}
\otimes a_{(2)})\wedge (g\otimes 1)\big) -f_{(-1)}\cdot
\left((f_{(2)}\otimes a)\wedge
(g\otimes 1)\right)\nonumber\\
&+ (ga_{(1)})_{(-1)}\cdot\big((f\otimes 1)\wedge ((ga_{(1)})_{(2)}
\otimes a_{(2)})\big) -g_{(-1)}\cdot \left((f\otimes
1\right)\wedge(g_{(2)}\otimes a)),
\end{align*}
gives  a  Lie pseudo-bialgebra structure on $gc_1$. This is an
example of coboundary  Lie pseudo-bialgebra defined in the
following section and it is a generalization of an example given
in \cite{L}.
\end{example}

\begin{remark} The examples presented here shows that this theory
is richer than the classical Lie bialgebra theory. We are far from
classification results in this context. Observe that it is not
known
 if a conformal version or a pseudoalgebra version
  of Whitehead's lemma holds for Cur$(\fg)$.
\end{remark}

\vskip 1cm

\section{Coboundary  Lie pseudo-bialgebras}\label{sec:cob}

\vskip .5cm

In this section we study a very important class of  Lie
pseudoalgebras, for which the $H$-coalgebra structure comes from a
1-coboundary.

\subsection{Cohomology of pseudoalgebras}\label{subcohom}

For the sake of completeness, we shall review some of the
definition given in Section 15 of \cite{BDK} .

\subsubsection{The complexes $C^\bullet(L,M)$ } \lbb{sclm}
 As
before, $H$ is a cocommutative Hopf algebra. Let $L$ be a Lie
$H$-\psalg\ and $M$ be an $L$-module.

By definition, $C^n(L,M)$, $n\ge1$, consists of all
\begin{equation}\lbb{nga}
\ga\in \Hom_{H^{\tt n}} (L^{\tt n}, H^{\tt n}\tt_H M)
\end{equation}
that are skew-symmetric. Explicitly, $\ga$ has the following
defining properties (cf.\ \eqref{bil*}, \eqref{ssym*}):

\noindent {\bf $H$-polylinearity:} For $h_i\in H$, $a_i\in L$,
\begin{equation*}
\ga(h_1 a_1 \tt\dotsm\tt h_n a_n) = ((h_1 \tt\dotsm\tt h_n)\tt_H
1)\,  \ga(a_1 \tt\dotsm\tt a_n).
\end{equation*}

\noindent{\bf Skew-symmetry:}
\begin{equation*}
\begin{split}
\ga(a_1 \tt\dotsm &\tt a_{i+1}\tt a_i \tt\dotsm\tt a_n)
\\
&= -(\si_{i,i+1} \tt_H\id)\, \ga(a_1 \tt\dotsm\tt a_i\tt a_{i+1}
\tt\dotsm\tt a_n),
\end{split}
\end{equation*}
where $\si_{i,i+1}\colon H^{\tt n} \to H^{\tt n}$ is the
transposition of the $i$th and $(i+1)$st factors.

For $n=0$, we put $C^0(L,M) = \kk\tt_H M \simeq M/H_+M$, where
$H_+ = \{ h\in H \st \ep(h)=0 \}$ is the augmentation ideal. The
differential $d\colon C^0(L,M) = \kk\tt_H M \to C^1(L,M) =
\Hom_H(L,M)$ is given by:
\begin{equation}\lbb{diga0}
\bigl( d(1\tt_H m) \bigr)(a) = \tsum_i \, (\id\tt\ep)(h_i) \, m_i
\in M
\end{equation}
if $a*m = \tsum_i \, h_i \tt_H m_i \in H^{\tt2} \tt_H M$, for
$a\in L$, $m\in M$.

For $n\ge1$, the differential $d\colon C^n(L,M) \to C^{n+1}(L,M)$
is given by
\begin{equation}\lbb{diga}
\begin{split}
&(d\ga)(a_1\tt\dotsm\tt a_{n+1})
\\
& = \;\; \sum_{1\le i\le n+1} (-1)^{i+1} (\si_{1 \to i}
\tt_H\id)\, a_i * \ga(a_1 \tt\dotsm\tt {\what a}_i \tt\dotsm\tt
a_{n+1})
\\
& + \sum_{1\le i<j \le n+1} (-1)^{i+j} (\si_{1 \to i, \, 2 \to j}
\tt_H\id)
\\
&\qquad\qquad\qquad\times \ga([a_i * a_j] \tt a_1 \tt\dotsm\tt
        {\what a}_i \tt\dotsm\tt {\what a}_j \tt\dotsm\tt a_{n+1}),
\end{split}
\end{equation}
where $\si_{1 \to i}$ is the permutation $h_i \tt h_1 \tt\dotsm\tt
h_{i-1}\tt h_{i+1} \tt\dotsm\tt h_{n+1} \mapsto h_1 \tt\dotsm\tt
h_{n+1}$, and $\si_{1 \to i, \, 2\to j}$ is the permutation $h_i
\tt h_j \tt h_1 \tt\dotsm\tt h_{i-1}\tt h_{i+1} \tt\dotsm\tt
h_{j-1}\tt h_{j+1} \tt\dotsm\tt h_{n+1} \mapsto h_1 \tt\dotsm\tt
h_{n+1}$.

In \eqref{diga} we also use the following conventions. If $a*b =
\tsum_i\, f_i \tt_H c_i \in H^{\tt 2}\tt_H M$ for $a\in L$, $b\in
M$, then for any $f\in H^{\tt n}$ we set:
\begin{displaymath}
a*(f\tt_H b) = \tsum_i\, (1\tt f) \, (\id\tt\De^{(n-1)})(f_i)
\tt_H c_i \in H^{\tt (n+1)} \tt_H M,
\end{displaymath}
where $\De^{(n-1)} =
(\id\tt\dotsm\tt\id\tt\De)\dotsm(\id\tt\De)\De \colon H \to H^{\tt
n}$ is the iterated comultiplication ($\De^{(0)} := \id$).
Similarly, if $\ga(a_1 \tt\dotsm\tt a_n) = \tsum_i\, g_i \tt_H v_i
\in H^{\tt n} \tt_H M$, then for $g\in H^{\tt 2}$ we set:
\begin{multline*}
\ga((g\tt_H a_1) \tt a_2 \tt\dotsm\tt a_n)
\\
= \tsum_i\, (g\tt 1^{\tt(n-1)}) \, (\De\tt\id^{\tt(n-1)})(g_i)
\tt_H v_i \in H^{\tt (n+1)} \tt_H M.
\end{multline*}
These conventions reflect the compositions of polylinear maps in
$\M^*(H)$, see \eqref{compoly3}. Note that \eqref{diga} holds also
for $n=0$ if we define $\De^{(-1)} := \ep$.

The fact that $d^2=0$ is most easily checked using
 the same argument as in the usual Lie algebra
case. The cohomology of the resulting complex $C^\bullet(L,M)$ is
called the {\em reduced cohomology of $L$ with coefficients in
$M$\/} and is denoted by $\coh^\bullet(L,M)$, (cf.\ \cite{BKV}).

\begin{remark} Note that the cocycle condition for the
cocommutator $\delta:L\to \wedge^2 L$ in the definition of  a
$H$-coalgebra is indeed the condition that $\delta$ is a 1-cocycle
of $L$ with coefficients in $\wedge^2 L$ in the reduced complex.
\end{remark}

\subsection{Definitions and conformal CYBE}\label{subcybe}

Among all the 1-cocycles of $L$ with values in $\wedge^2 L$, we
have  1-coboundaries $\delta$ that comes from the differential of
an element $r\in\wedge^2 L$, that is
\begin{equation}\label{eq:deltar}
\delta_r(a)= \left(d (1\otimes_H r)\right) (a),
\end{equation}
for all $a\in L$, cf. (\ref{diga0}).

\begin{definition} A {\it coboundary  Lie pseudo-bialgebra} is a triple $(L, [ \ * \ ], r)$,
 with $r\in L\otimes L$, such that  $(L, [ \ * \ ], \delta_r)$ is a  Lie pseudo-bialgebra.
 In this case, the element $r\in L\otimes L$ is said to be  a {\it coboundary structure}.
\end{definition}

Now, we can state one of the main results of this article.

\begin{theorem}\label{teo:dual} Let $L$ be a Lie pseudoalgebra
 and   $\mu: H\otimes (L\otimes L)\longrightarrow L\otimes L$  given by
$\mu(h\otimes m\otimes n)=\Delta (h)(m\otimes n)$. Let $r=\sum_i
a_i\otimes b_i\in L\otimes L$. The map $\delta_r:L\longrightarrow
L\otimes L$ given by ($a\in L$)
$$
\delta_r(a)=(d(1\otimes_H \,r))(a)=\sum_i \mu([a,a_i]\otimes
b_i+\sigma_{12}(a_i\otimes [a,b_i])),
$$
 is the
cocommutator of a Lie pseudo-bialgebra structure on $L$ if and
only if the following conditions are satisfied: \vskip .3cm

\noindent (1) the symmetric part of $r$ is $L$-invariant, that is:
$$
\delta_{r+r_{21}}(a)=0,
$$
where $r_{12}=\sum_i b_i\otimes a_i$. \vskip .3cm

\noindent (2) $$\mu_3(a\cdot[[r,r]])=0$$ where $\mu_3(h\otimes
m\otimes n\otimes p)=((\Delta\otimes 1)\Delta (h)) (m\otimes
n\otimes p)$, the dot action is the action analogous to the
bracket defined in (\ref{ab}) and

\begin{equation}\label{[[r,r]]}
[[r,r]]=\mu_{-1}^3([a_j,a_i]\otimes b_j\otimes b_i)
-\mu_{-2}^4(a_i\otimes[a_j,b_i]\otimes b_j) -\mu_{-3}^2(a_i\otimes
a_j\otimes[b_j,b_i]),
\end{equation}
\vskip .2cm

\noindent where $\mu_{-k}^l$ means that the element of $H$ that
appears in its argument in the $k$-th place acts via the antipode
on the element of $L$ located in the $l$-th entry.
\end{theorem}

\begin{proof}From now on we will use the following notation: For
$a$ and $b$ in $L$,
\begin{equation*}
 [a*b]=(h^{a,b}\otimes 1)\tt_H c_{a,b}\,\in (H\tt H)\tt_H L.
\end{equation*}
Since $r=\sum_i a_i\otimes b_i\in L\otimes L$, by
(\ref{eq:tensor})
\begin{eqnarray*}
a *  r= (h^{a,a_i}\otimes 1)\tt_H (c_{a,a_i}\tt b_i) +
(h^{a,b_i}\otimes 1)\tt_H (a_i\tt c_{a,b_i}).
\end{eqnarray*}
Thus
\begin{align}\label{jose}
\delta_r(a)&=(d(1\tt_Hr))(a)\nonumber\\
&= \left(h^{a,a_i}\cdot ( c_{a,a_i}\tt b_i) + h^{a,b_i}\cdot(
a_i\tt c_{a,b_i})\right)\nonumber\\
&=\mu\left( [a,a_i]\otimes b_i+\sigma_{12} (a_i\tt[a,b_i])\right),
\end{align}
where we set $\mu(h\otimes m\otimes n)=\Delta (h) (m\otimes n)$
for all $h\in H$ and $m$ and $n$ in $L$, and $[a,b]\in H\tt L$ is
the Fourier transform of $[a*b]$ (cf. (\ref{ab})).

It is clear that the skew-symmetry of $\delta$ is equivalent to
condition (1) in the statement of the theorem. Now
\begin{align*}
(\delta_r\tt \id)&\delta_r(a)=(\delta_r\tt \id)
\left(h^{a,a_i}\cdot ( c_{a,a_i}\tt b_i) + h^{a,b_i}\cdot( a_i\tt
c_{a,b_i})\right)\nonumber\\
&=  h_{(1)}^{a,a_i}\,\delta_r(c_{a,a_i})\tt h_{(2)}^{a,a_i}\,b_i +
h_{(1)}^{a,b_i}\delta_r( a_i)\tt
h_{(2)}^{a,b_i}c_{a,b_i}\nonumber\\
&=h^{a,a_i}\cdot(\delta_r(c_{a,a_i})\tt \,b_i) +
h^{a,b_i}\cdot(\delta_r( a_i)\tt
c_{a,b_i})\nonumber\\
&=h^{a,a_i}\cdot\left(h^{c_{a,a_i},a_j}\cdot(c_{c_{a,a_i},a_j}\tt
\,b_j)\tt b_i +h^{c_{a,a_i},b_j}\cdot(a_j\tt c_{c_{a,a_i},b_j})
\tt b_i\right)\nonumber\\
 &\qquad+ h^{a,b_i}\cdot\left(h^{a_i,a_j}\cdot(c_{a_i,a_j}\tt
\,b_j)\tt c_{a,b_i}+h^{a_i,b_j}\cdot(a_j\tt c_{a_i,b_j})
\tt c_{a,b_i}\right)\nonumber\\
&=\mu_1(\mu_2^{3,4}\left([[a,a_i],a_j]\tt b_j\tt b_i\right))
+\mu_2(\mu_3^{1,4}\left(a_j\tt[[a,a_i],b_j]\tt  b_i\right))\nonumber\\
&\qquad+\mu_3(\mu_1^{2,3}\left([a_i,a_j]\tt b_j\tt [a,b_i]\right))
+\mu_3(\mu_2^{1,3}\left(a_j\tt [a_i,b_j]\tt [a,b_i])\right),
\end{align*}
where $\mu_{k}^{r,s}$ means that the element of $H$ that appears
in its argument in the $k$-th place acts on the elements of $L\tt
L$ formed by the elements  in the $r$ and $s$-th entries, and then
relocated in its original places, omitting $k$-th place. For
example $\mu_{2}^{1,4}(m\tt f\tt n\tt p\tt g)=f_{(1)}m\tt n\tt
f_{(2)}p\tt g$, with $m$, $n$ and $p$ in $L$ and $f$, $g$ in $H$.
Similarly, $\mu_k$ represents the action of the element of $H$ in
the $k$-th place acting in the element of $L\tt L\tt L$ formed by
the remaining elements in its argument.

Now, we can write down the twelve terms in
$\sum_{c.p.}(\delta_r\tt \id)\delta_r(a)$, where $\sum_{c.p.}$
stands for cyclic permutations of the factors in $L\tt L\tt L$.
Namely
\begin{align}
\hbox{$\sum_{c.p.}$}(\delta_r\tt \id)\delta_r(a)&=\mu_1(\mu_2^{3,4}\left([[a,a_i],a_j]\tt b_j\tt b_i\right))\label{(1)}\\
&\qquad\qquad\qquad+\mu_2(\mu_3^{1,4}\left(a_j\tt[[a,a_i],b_j]\tt  b_i\right))\label{(2)}\\
&\qquad+\mu_3(\mu_1^{2,3}\left([a_i,a_j]\tt b_j\tt [a,b_i]\right))\label{(3)}\\
&\qquad\qquad\qquad+\mu_3(\mu_2^{1,3}\left(a_j\tt [a_i,b_j]\tt [a,b_i])\right)\label{(4)}\\
&\qquad+\mu_2(\mu_3^{4,5}\left(b_i\tt[[a,a_i],a_j]\tt b_j\right))\label{(9)}\\
&\qquad\qquad\qquad+\mu_3(\mu_4^{2,5}\left(b_i\tt a_j\tt[[a,a_i],b_j]\right))\label{(10)}\\
&\qquad+\mu_1(\mu_3^{4,5}\left([a,b_i]\tt [a_i,a_j] \tt b_j\right))\label{(11)}\\
&\qquad\qquad\qquad+\mu_1(\mu_4^{3,5}\left([a,b_i]\tt a_j \tt[a_i,b_j] )\right)\label{(12)}\\
&\qquad+\mu_3(\mu_4^{1,5}\left(b_j\tt b_i\tt[[a,a_i],a_j]  \right))\label{(5)}\\
&\qquad\qquad\qquad+\mu_1(\mu_2^{3,5}\left([[a,a_i],b_j]\tt b_i\tt a_j \right))\label{(6)}\\
&\qquad+\mu_2(\mu_4^{1,5}\left( b_j\tt [a,b_i]\tt[a_i,a_j]\right))\label{(7)}\\
&\qquad\qquad\qquad+\mu_2(\mu_1^{2,5}\left( [a_i,b_j]\tt
[a,b_i]\tt a_j)\right)\label{(8)}
\end{align}

On the other hand, by the skew commutativity in (\ref{ss}) we have
that
\begin{align}
\mu_3(a\cdot[[r,r]])&=\mu_3\left(a\cdot\left(\mu_{-1}^3([a_j,a_i]\otimes
b_j\otimes b_i) \right.\right.\nonumber\\
&\qquad\qquad-\left.\left.\mu_{-2}^4(a_i\otimes[a_j,b_i]\otimes
b_j) -\mu_{-3}^2(a_i\otimes
a_j\otimes[b_j,b_i]\right)\right)\nonumber\\
&=-\mu_1\left(\mu_2^{3,4} (\mathcal{F} \tt \id\tt\id\tt\id)([a,[a_i,a_j]]\tt b_j\tt b_i)\right)\label{1tilde}\\
&\qquad\qquad+\mu_2\left(\mu_1^{2,5}([a_i,a_j] \tt [a,b_i]\tt
b_j)\right)\label{2tilde}\\
&\qquad +  \mu_3\left(\mu_{-1}^{3}([a_j,a_i]\tt
b_j\tt[a,b_i])\right)\label{3tilde}\\
&\qquad\qquad+\mu_1\left(\mu_3^{4,5}([a,a_i]\tt[b_i,a_j]\tt b_j)\right) \label{4tilde}\\
&\qquad+\mu_2\left(\mu_{3}^{4,5}(\id\tt\mathcal{F}\tt\id\tt\id)(a_j\tt[a,[b_j,a_i]]\tt
b_i)\right)\label{5tilde} \\
&\qquad\qquad+ \mu_3\left(\mu_{-2}^1(a_j\tt[b_j,a_i]\tt
[a,b_i])\right)\label{6tilde}\\
&\qquad+\mu_1\left(\mu_{4}^{3,5}([a,a_i]\tt a_j\tt
[b_i,b_j])\right)\label{7tilde} \\
 &\qquad\qquad- \mu_3^{14}\left(a_i\tt (\id\tt \rho(b_i))(\mu_{1}^{2,3}([a,a_j]\tt b_j)\right)\label{8tilde}\\
&\qquad+\mu_3\left(\mu_{4}^{2,5}(\id\tt\id\tt\mathcal{F}\tt\id)(a_i\tt
a_j\tt [a,[b_i,b_j]])\right)\label{9tilde},
\end{align}
where $\mathcal{F}$ is the Fourier Transform defined in
(\ref{ftrans}) and $\rho(b)(a)=[a,b]$ for all $a$ and $b$ in $L$.
Now, the study of the sum of both equations is divided in several
steps.

First, observe that (\ref{(3)})$ +$ (\ref{3tilde})$=0$. Indeed,
using the skew-symmetry introduced in (\ref{ss}) , if
$[a_j,a_i]=h^{a_j,a_i}\tt c_{a_j,a_i}$  we get
\begin{align*}
\mu_3\left(\mu_{1}^{2,3}([a_i,a_j]\tt b_j\tt
[a,b_i])\right)&=-\mu_3\left(\mu_{1}^{2,3}(h^{a_j,a_i}_{(-1)}\tt
h^{a_j,a_i}_{(2)}c_{a_j,a_i}\tt b_j\tt
[a,b_i])\right) \nonumber\\
&= -\mu_3\left(h^{a_j,a_i}_{(-1)_{(1)}}\, h^{a_j,a_i}_{(2)}\,
c_{a_j,a_i}\tt h^{a_j,a_i}_{(-1)_{(2)}}\, b_j\tt
[a,b_i]\right) \nonumber\\
&=-\mu_3\left( c_{a_j,a_i}\tt S(h^{a_j,a_i})\, b_j\tt [a,b_i]\right)\nonumber\\
&=-  \mu_3\left(\mu_{-1}^{3}([a_j,a_i]\tt b_j\tt[a,b_i])\right)
\end{align*}
Similarly, we have (\ref{(4)})$+$(\ref{6tilde})$=0$.

Interchanging the indices $i$ and $j$ and using Jacobi identity
(\ref{jacid}), we have that (\ref{(1)})+(\ref{1tilde}) is
\begin{align} \label{eq:1+1}
& \mu_1(\mu_2^{3,4}\left([[a,a_i],a_j]\tt b_j\tt b_i\right))
-\mu_1\left(\mu_2^{3,4} (\mathcal{F} \tt
\id\tt\id\tt\id)([a,[a_i,a_j]]\tt b_j\tt b_i)\right)\nonumber\\
&\qquad\qquad =-
\mu_1\left(\mu_2^{3,4}\left((\mathcal{F}\circ\sigma \tt
\id\tt\id\tt\id)([a_i,[a,a_j]]\tt b_j\tt b_i)\right)\right)\nonumber\\
&\qquad\qquad=-\mu_1\left( \mu_2^{3,5}([[a,a_j],a_i]\tt b_j\tt
b_i\right).
\end{align}

Now, using the invariance property of part (1) of this theorem,
and (\ref{eq:1+1}), we obtain that (\ref{(6)})+ (\ref{eq:1+1}) is

\begin{align*}
&\mu_1(\mu_2^{3,5}\left([[a,a_i],b_j]\tt b_i\tt a_j
\right))-\mu_1\left( \mu_2^{3,5}([[a,a_j],a_i]\tt b_j\tt
b_i\right)\nonumber\\
&\qquad\qquad\qquad= \mu_1\left(\mu_2^{3,5}\left( [[a,a_j],b_i]\tt
b_j\tt a_i+[[a,a_j],a_i]\tt b_j\tt
b_i\right)\right)\nonumber\\
&\qquad\qquad\qquad=-\mu_3\left(\mu_4^{1,5}(b_i\tt
b_j\tt[[a,a_j],a_i] + a_i\tt
b_j\tt[[a,a_j],b_i])\right)\nonumber\\
&\qquad\qquad\qquad:=(A)+(B).
\end{align*}

\noindent It is easy to see that $(A)$ +(\ref{(5)})$=0$, hence it
remains to cancel $(B)$.

Now, recall that  $\rho(x)(y)=[y,x]$ for all $x$ and $y$ in $L$,
then using again the invariance of part (1), we get that
(\ref{(11)})+(\ref{4tilde}) is
\begin{align*}
&\mu_1(\mu_3^{4,5}\left([a,b_i]\tt [a_i,a_j] \tt
b_j\right))+\mu_1\left(\mu_3^{4,5}([a,a_i]\tt[b_i,a_j]\tt
b_j)\right) \nonumber\\
&\qquad\qquad\qquad=\mu_1\left(\mu_3^{4,5}(\id\tt\rho(a_j)\tt\id)([a,b_i]\tt
a_i\tt
b_j+[a,a_i]\tt b_i\tt b_j)\right)\nonumber\\
&\qquad\qquad\qquad=\mu_2^{3,4}\left\{\Big((\id\tt\rho(a_j))\mu_1^{2,3}\big([a,b_i]\tt
a_i+[a,a_i]\tt b_i\big)\Big)\tt b_j \right\}\nonumber\\
&\qquad\qquad\qquad=-\mu_2^{3,4}\left\{\Big((\id\tt\rho(a_j))\mu_2^{1,3}\big(b_i\tt
[a,a_i]+a_i\tt [a,b_i]\big)\Big)\tt b_j \right\}\nonumber\\
&\qquad\qquad\qquad=-\mu_2\left(\mu_3^{4,5}(b_i\tt[[a,a_i],a_j]\tt
b_j+a_i\tt[[a,b_i],a_j]\tt b_j)\right)\nonumber\\
&\qquad\qquad\qquad:=(D)+(C).
\end{align*}
It is obvious that $(D) +(\ref{(9)})=0$, hence it remains $(C)$.

Similarly, we have that (\ref{(12)})+(\ref{7tilde}) is

\begin{align*}
&\mu_1(\mu_4^{3,5}\left([a,b_i]\tt a_j \tt[a_i,b_j]
)\right)+\mu_1(\mu_4^{3,5}\left([a,a_i]\tt a_j\tt
[b_i,b_j]\right))\nonumber\\
&\qquad\qquad=\mu_1\left(\mu_4^{3,5}(\id\tt\id\tt\id\tt\rho(b_j))([a,b_i]\tt
a_j\tt
a_i+[a,a_i]\tt a_j\tt b_i)\right)\nonumber\\
&\qquad\qquad=\mu_3^{2,4}\left\{\Big((\id\tt\id\tt\rho(b_j))\mu_1^{2,4}\big([a,b_i]\tt
a_j\tt a_i+[a,a_i]\tt a_j\tt b_i\big)\Big)\right\}\nonumber\\
&\qquad\qquad=-\mu_3^{2,4}\left\{\Big((\id\tt\id\tt\rho(b_j))\mu_3^{1,4}\big(b_i\tt
a_j\tt[a.a_i]+a_i\tt a_j\tt [a,b_i]\big)\Big) \right\}\nonumber\\
&\qquad\qquad=-\mu_3\left(\mu_4^{2,5}(b_i\tt a_j\tt
[[a,a_i],b_j]+ a_i\tt a_j\tt [[a,b_i]b_j)]\right)\nonumber\\
&\qquad\qquad:=(F)+(E).
\end{align*}
and it is obvious that $(F) + (\ref{(10)})=0$, hence it remains
$(E)$. In a similar  way, it is easy to see that
\begin{align*}
\hbox{(\ref{(8)})+(\ref{2tilde})}&=
\mu_2\left(\mu_1^{2,5}\big([a_i,b_j]\tt[a,b_i]\tt a_j +
[a_i,a_j]\tt[a,b_i]\tt b_j\big)\right) \nonumber\\
&=-\mu_2\left(\mu_4^{1,5}\big(a_j\tt[a,b_i]\tt[a_i,b_j]+ b_j\tt
[a,b_i]\tt[a_i,a_j]\big)\right) \nonumber\\
&:=(H)+(G),
\end{align*}
and we have $(G)$+(\ref{(7)})$=0$, hence it remains $(H)$.

By a simple computation  it is easy to see that
(\ref{(2)})+(\ref{5tilde})+ $(C)=0$ by Jacobi identity
(\ref{jacid}). Now, we can write, using skew-symmetry and
invariance property,
\begin{align}\label{eq:8+A+H}
\hbox{(\ref{8tilde})}+(B)+(H)=\mu_3\Big(a_i\tt(\id\tt\rho(b_i))\mu_2^{1,3}(a_j\tt[a,b_j])\Big)
\end{align}

Finally, a simple computation shows that (\ref{9tilde}) + $(E)$ +
(\ref{eq:8+A+H})$\,=0 $ by Jacobi identity, and it is easy to
check that we have cancelled all the terms, finishing the proof.
\end{proof}

\begin{definition} \label{def:quasi} A {\it quasitriangular  Lie pseudo-bialgebra} is a
coboundary Lie bialgebra  $(L, [\,*\, ] , r)$ with $r\in L\otimes
L$ such that

\vskip .3cm

\begin{enumerate}
\item $[[r,r]]=0,\;$
mod$(H_+\cdot(L\tt L\tt L))$, where $H_+$ is the augmentation
ideal,

\,

\item $r$ is $L$-invariant, namely $\delta_{r+r_{21}}(a)=0$.
 \end{enumerate}
\end{definition}

\vskip 1cm

\section{Pseudo Manin triples}\label{submanin}

\vskip .5cm

Let $V$ be a $H$-module. A {\it  bilinear pseudo-form} on $V$ is a
$\mathbf{k}$-bilinear map $\< \  , \  \> : V\times V\to (H\otimes
H)\otimes_H \kk$ such that
\begin{eqnarray*}
\<h v,w\> &=& ((h\otimes 1)\otimes_H 1)\< v,  w\> \\
 \< v,hw\> &=& ((1\otimes h)\otimes_H 1)\< v,  w\>\qquad\hbox{ for all
} v,w\in V, h\in H.
\end{eqnarray*}

\noindent We call a  bilinear pseudo-form  {\it symmetric }  if
\begin{equation*}
\< v,w\>= (\sigma\otimes_H 1)   \< w , v\>\qquad\hbox{ for all }
v,w\in V.
\end{equation*}

\noindent A  bilinear pseudo-form in a  Lie pseudoalgebra $L$ is
called {\it invariant} if
\begin{equation} \label{eq:inv}
\< [a* b],c \>=    \< a , [b*c]\>
\end{equation}
for all $a,b,c\in L$, where the usual composition rules of
polylinear maps are applied in (\ref{eq:inv}).

\vskip .2cm

Given a   bilinear pseudo-form on a $H$-module $V$, we have a
homomorphism of  $H$-modules,  $\phi : V\to
V^{*}=\hbox{Chom}(V,\mathbf{k})$, $v\mapsto \phi_v$, given as
usual by
\begin{equation*}
(\phi_v)* w= \< v, w\>,  \qquad v\in V.
\end{equation*}

Now, suppose that a bilinear pseudo-form satisfies that  $\< v,w
\> =0$ for all $w\in V$, implies $v=0$. Then $\phi$ gives an
injective map between $V$ and $ V^{*}$, but not necessarily
surjective.

Following \cite{L}, a
 bilinear pseudo-form is called {\it
non-degenerate} if $\phi$ gives an isomorphism between $V$ and
$V^{*}$.

\begin{definition} A (finite rank) {\it pseudo Manin triple} is a triple of finite
rank Lie pseudoalgebras $(L, L_1, L_0)$, where $L$ is equipped
with a non-degenerate invariant symmetric bilinear pseudo-form
$\<\ ,\ \>$ such that

\vskip .2cm

1. $L_1, L_0$ are Lie pseudosubalgebras of $L$ and $L=L_0\oplus
L_1$ as $H$-module.

\vskip .2cm

2. $L_0$ and $L_1$ are isotropic with respect to $\<\ ,\ \>$, that
is $\<L_i ,L_i \>=0$ for $i=0,1$.
\end{definition}

\begin{theorem} \label{th:manin}
Let $L$ be a Lie pseudoalgebra free of finite rank. Then there is
a one-to-one correspondence between Lie pseudo-bialgebra
structures on $L$ and pseudo Manin triples $(R, R_1, R_0)$ such
that $R_1=L$.
\end{theorem}

\begin{proof} Given  a Lie  pseudo-bialgebra $L$, we construct a pseudo
Manin triple in the following way: we set $R_1=L$, $R_0=L^{*}$
with the Lie pseudoalgebra structure given by the dual of the
coalgebra structure in $L$, $R=L\oplus L^{*}$, and take the
non-degenerate symmetric  bilinear pseudo-form in $R$ given by
\begin{displaymath}
\< a+f , b+g \> =f*(b)+ (\sigma\tt_H\id)g*(a),
\end{displaymath}
for all $a,b\in L$ and $f,g\in L^*$. Now, observe that the
invariance of the bilinear form uniquely determines the bracket on
$L\oplus L^{*}$, namely: let $\{e_i\}_{i=1}^n$ be an $H$-basis of
$R_1$ and let $\{e_i^*\}_{i=1}^n$ be the corresponding dual basis
in $R_0\simeq R_1^{*}$ and set
\begin{equation*}
[e_i*e_j]=\sum_k (h_k^{ij}\tt l_k^{ij})\tt_H e_k\quad\hbox{ and }
\quad [e_i^**e_j^*]=\sum_k (f_k^{ij}\tt g_k^{ij})\tt_H e_k^*.
\end{equation*}
Due to the invariance of the bilinear form, we have
\begin{align*}
\langle[e_i^**e_j],e_l\rangle&= \langle e_i^*,[e_j*e_l]\rangle
=\sum _k (1\tt h_k^{jl}\tt l_k^{jl})\tt_H \delta_{ik}\nonumber\\
&=(1\tt h_i^{jl}\tt l_i^{jl})\tt_H 1 = \left((l_i^{jl})_{(-2)}\tt
h_i^{jl}(l_i^{jl})_{(-1)}\tt 1\right)\tt_H
1\nonumber\\
&=\left\langle  \sum_s ((l_i^{js})_{(-2)}\tt
h_i^{js}(l_i^{js})_{(-1)})\tt_H e^*_s, e_l \right\rangle,
\end{align*}
and
\begin{align*}
\langle[e_j*e_i^*],e_k^*\rangle&= \langle e_j,[e_i^**e_k^*]\rangle
=\sum _l (1\tt f_l^{ik}\tt g_l^{ik})\tt_H \delta_{jl}\nonumber\\
&=(1\tt f_j^{ik}\tt g_j^{ik})\tt_H 1 =\left((g_j^{ik})_{(-2)}\tt
f_j^{ik} (g_j^{ik})_{(-1)}\tt 1\right)\tt_H
1\nonumber\\
&=\left\langle \sum_s ((g_j^{is})_{(-2)}\tt f_j^{is}
(g_j^{is})_{(-1)})\tt_H e_s,e_k^*\right\rangle.
\end{align*}
Hence, using skew-symmetry, we have
\begin{align}\label{eq: manin-bkt}
[e_i^**e_j]&= \sum_s ((l_i^{js})_{(-2)}\tt
h_i^{js}(l_i^{js})_{(-1)})\tt_H e^*_s\\
&\hskip .5cm -\sum_r (f_j^{ir}
(g_j^{ir})_{(-1)}\tt(g_j^{ir})_{(-2)}) \tt_H e_r.\nonumber
\end{align}
It remains to show that this is indeed a Lie pseudoalgebra
 bracket. Let us first check Jacobi identity, namely we have to
 show that
 \begin{align*}
 0=[e_p^**[e_i*e_j]]-((\sigma\tt\id)\tt_H\id)[e_i*[e_p^**e_j]]-[[e_p^**e_i]*e_j].
 \end{align*}
 together with a similar relation involving two $e^*\,'s$ and one $e$. Expanding it,
 using (\ref{eq: manin-bkt}) and the composition rules
(\ref{abc*3'}) and (\ref{abc*6'}), we get
\begin{align}\label{eq: manin-jacobi}
0&=  \sum_{s,k} \left[(l_p^{ks})_{(-1)}\tt
h_k^{ij}(h_p^{ks}(l_p^{ks})_{(-2)})_{(1)} \tt
l_{k}^{ij}(h_p^{ks}(l_p^{ks})_{(-2)})_{(2)}\right]\tt_H e_s^*\nonumber\\
& -\sum_{k,r} \left[f_k^{pr}(g_k^{pr})_{(-1)}\tt
h_k^{ij}{(g_k^{pr})_{(-2)}}_{(1)}\tt
l_k^{ij}{(g_k^{pr})_{(-2)}}_{(2)}\right]\tt_H e_r\nonumber\\
&+\sum_{s,n} \left[(l_p^{js})_{(-1)}{(l_s^{in})_{(-1)}}_{(1)}\tt
h_s^{in}(l_s^{in})_{(-2)} \tt
h_p^{js}(l_p^{js})_{(-2)}{(l_{s}^{in})_{(-1)}}_{(2)}\right]\tt_H e_n^*\nonumber\\
&-\sum_{m,s}\left[(l_p^{js})_{(-1)}
(f_i^{sm}(g_i^{sm})_{(-1)})_{(1)}\tt (g_i^{sm})_{(-2)}\tt
h_p^{js}(l_p^{js})_{(-2)}(f_i^{sm}(g_i^{sm})_{(-1)})_{(2)}\right]\tt_H
e_m\nonumber\\
&+\sum_{r,m} \left[f_j^{pr}(g_j^{pr})_{(-1)}(l_m^{ir})_{(1)}\tt
h_m^{ir}\tt (g_j^{pr})_{(-2)}(l_m^{ir})_{(2)}\right]\tt_H e_m \nonumber\\
&+\sum_{s,n} \left[(l_p^{is})_{(-1)}{(l_s^{jn})_{(-1)}}_{(1)}\tt
h_p^{is}(l_p^{is})_{(-2)}{(l_{s}^{jn})_{(-1)}}_{(2)}\tt
h_s^{jn}(l_s^{jn})_{(-2)}\right] \tt_H e_n^*\nonumber\\
&-\sum_{m,s}\left[(l_p^{is})_{(-1)}(f_j^{sm}(g_j^{sm})_{(-1)})_{(1)}\tt
h_p^{is}(l_p^{is})_{(-2)}(f_j^{sm}(g_j^{sm})_{(-1)})_{(2)}\tt
(g_j^{sm})_{(-2)} \right] \tt_H
e_m\nonumber\\
&+\sum_{r,m} \left[ f_i^{pr}(g_i^{pr})_{(-1)}(l_m^{jr})_{(1)}\tt
 (g_i^{pr})_{(-2)}(l_m^{jr})_{(2)}\tt h_m^{jr}\right]\tt_H e_m.\nonumber\\
\end{align}

The coefficients of $e^*$  in (\ref{eq: manin-jacobi}) gives a
relation equivalent to the Jacobi identity of $L$, and it is easy
to see (after renaming some variables) that the coefficients of
$e$ in (\ref{eq: manin-jacobi}) gives a relation equivalent to
(\ref{eq:cojacobi-manin}) which is up to the identification
\begin{align}\label{identif-manin}
H\tt H\tt H\tt_H H &\longrightarrow H\tt H\tt_H H\tt H \nonumber\\
f\tt l\tt h\tt_H g &\;\mapsto\; l\tt h\tt_H S(f)\tt g
\end{align}
the 1-cocycle condition of the cobracket in $L$ (see below
(\ref{eq:cojacobi-manin})). In a similar way, the other Jacobi
identity in $L\oplus L^*$ is equivalent to
(\ref{eq:cojacobi-manin}) and the Jacobi identity of $L^*$.

\

Conversely, let $(R,R_1,R_0)$ be a pseudo Manin triple. The
non-degenerate pseudo-form $\<\ ,\ \>$ induces a non-degenerate
pairing $R_0\otimes R_1\to H$ that produce  an isomorphism
$R_1^{*}\simeq R_0$ as $H$-modules, and hence a Lie pseudoalgebra
structure on $R_1^{*}$. Denote by $\delta$ the Lie coalgebra
structure induced on $R_1$ by Theorem \ref{duality-thm}. We have
to show that $(R_1, [\ * \ ], \delta)$ is a Lie pseudo-bialgebra
and hence $R_0$ is its dual Lie pseudo-bialgebra. Therefore, we
have to check the cocycle condition
\begin{equation}\label{eq:m7}
0= \delta([a* b])-(\sigma\tt_H\id)\,(b*\delta (a))-a*\delta (b) .
\end{equation}
In order to do this, let $\{e_i\}_{i=1}^n$ be an $H$-basis of
$R_1$ and let $\{e_i^*\}_{i=1}^n$ be the dual basis in $R_0\simeq
R_1^{*}$. Set, as before,
\begin{displaymath}
[e_i \ *\  e_j]=\sum_s (h_s^{ij}\tt l^{ij}_s)\tt_H e_s,
\quad\hbox{and} \qquad [e_i^* \ * \ e_j^*]=\sum_s (f_s^{ij}\tt
g^{ij}_s)\tt_H e_s^*.
\end{displaymath}
By definition (see Theorem~\ref{duality-thm}),
\begin{displaymath}
\delta(e_k)=\sum_{i,j}  S(h_k^{ij})e^i\otimes S(l_k^{ij})e^j .
\end{displaymath}
Thus, we have
\begin{align}\label{manin:10}
\delta([e_j*e_l])&=\sum_{r,s,t} (h^{jl}_r\tt l^{jl}_r )\tt_H
(S(f_r^{st})e_s\tt S(g_r^{st})e_t)\nonumber\\
&= \sum_{r,s,t}\big((h^{jl}_r\tt l^{jl}_r )\tt_H (S(f_r^{st})\tt
S(g_r^{st}))\big) \left((1\tt 1)\tt_H (e_s\tt e_t)\right).
\end{align}
On the other hand
\begin{align}\label{manin:11}
&e_j*\delta(e_l)=\sum_{k,t,s}
\left(h_k^{js}\left(S(f_l^{st})l_k^{js}\right)_{(-1)}\tt
1\right)\tt_H \left(\left(
S(f_l^{st})l_k^{js}\right)_{(2)}e_k\tt S(g_l^{st}) e_t\right)\nonumber\\
&\qquad\qquad\qquad + \sum_{n,t,s}
\left(h_n^{jt}\left(S(g_l^{st})l_n^{jt}\right)_{(-1)}\tt
1\right)\tt_H \left(S(f_l^{st}) e_s \tt \left(
S(g_l^{st})l_n^{jt}\right)_{(2)}e_n\right)\nonumber\\
&\qquad=\sum_{k,t,s} \Big(
h_k^{js}\left(S(f_l^{st})l_k^{js}\right)_{(-1)}\tt
1\Big)\tt_H\Big(\left( S(f_l^{st})l_k^{js}\right)_{(2)}e_k \tt
S(g_l^{st})e_t\Big)\\
& \qquad\qquad + \sum_{n,t,s}
\left(h_n^{jt}\left(S(g_l^{st})l_n^{jt}\right)_{(-1)}\tt
1\right)\tt_H \left(S(f_l^{st}) e_s \tt \left(
S(g_l^{st})l_n^{jt}\right)_{(2)}e_n\right),\nonumber
\end{align}
and
\begin{align} \label{manin:12}
&(\sigma\tt_H\id)(e_j*\delta(e_j))=\nonumber\\
&=\sum_{k,t,s} \Big(1 \tt
h_k^{ls}\left(S(f_j^{st})l_k^{ls}\right)_{(-1)}
\Big)\tt_H\Big(\left(
S(f_j^{st})l_k^{js}\right)_{(2)}e_k\tt S(g_j^{st}) e_t\Big)\nonumber\\
&\qquad + \sum_{n,t,s} \left(1 \tt
h_n^{lt}\left(S(g_j^{st})l_n^{lt}\right)_{(-1)} \right)\tt_H
\left(S(f_j^{st}) e_s \tt \left(
S(g_j^{st})l_n^{lt}\right)_{(2)}e_n\right)\nonumber\\
&=\sum_{k,t,s}\Big( 1 \tt
h_k^{ls}\left(S(f_j^{st})l_k^{js}\right)_{(-1)}
\Big)\tt_H\Big(\left( S(f_j^{st})l_k^{ls}\right)_{(2)}e_k\tt
S(g_j^{st})e_t\Big)\nonumber\\
&\qquad + \sum_{n,t,s} \Big(1\tt
h_n^{lt}\left(S(g_j^{st})l_n^{lt}\right)_{(-1)}\Big) \tt_H \Big(
S(f_j^{st})e_s \tt \left(
S(g_j^{st})l_n^{lt}\right)_{(2)}e_n\Big).
\end{align}
By taking the coefficients of $(1\tt 1)\tt_H (e_p \otimes e_q)$ in
(\ref{manin:10}), (\ref{manin:11}) and (\ref{manin:12}), the
cocycle condition (\ref{eq:m7}) become (after renaming subindexes)
\begin{align}\label{eq:cojacobi-manin}
\sum_{r} (h^{jl}_r\tt l^{jl}_r )\tt_H &(S(f_r^{pq})\tt
S(g_r^{pq}))=
\\
&\, =\sum_{s}
\big(h_p^{js}\left(S(f_l^{sq})l_p^{js}\right)_{(-1)}\tt
1\big)\tt_H \big(\left(
S(f_l^{sq})l_p^{js}\right)_{(2)}\tt S(g_l^{sq})\big)\nonumber\\
&\, + \sum_{r}
\big(h_t^{jr}\left(S(g_l^{pr})l_t^{jr}\right)_{(-1)}\tt
1\big)\tt_H \big(S(f_l^{pr})  \tt \left(
S(g_l^{pr})l_t^{jr}\right)_{(2)}\big)\nonumber\\
&\,-\sum_{s} \big(1 \tt
h_p^{ls}\left(S(f_j^{sq})l_p^{js}\right)_{(-1)}\big) \tt_H
\big(\left(
S(f_j^{sq})l_p^{ls}\right)_{(2)}\tt S(g_j^{sq})\big)\nonumber\\
&\, - \sum_{r} \big(1\tt
h_i^{lr}\left(S(g_j^{pr})l_i^{lr}\right)_{(-1)}\big) \tt_H
\big(S(f_j^{pr}) \tt \left(
S(g_j^{pr})l_i^{lr}\right)_{(2)}\big).\nonumber
\end{align}
which is equivalent under the identification (\ref{identif-manin})
to the coefficients of $e_t$ in (\ref{eq: manin-jacobi}), that is,
the Jacobi identity on $R=R_1\oplus R_0\simeq R_1\oplus R_1^{*}$,
finishing the proof.
\end{proof}

\vskip 1cm

\section{Pseudo Drinfeld's double}\label{subdrinf}

\vskip .5cm

The correspondence between pseudo-bialgebras and pseudo Manin
triples gives us a Lie pseudoalgebra structure on $L\oplus L^{*}$
if $L$ is a pseudo-bialgebra. In fact, a more general result is
true.

\begin{theorem} Let $L$ be a free finite rank Lie pseudo-bialgebra
and let\break $(L\oplus L^{*}, L , L^{*})$ be the associated
pseudo Manin triple. Then there is a canonical  Lie
pseudo-bialgebra structure on $L\oplus L^{*}$ such that the
inclusions
\begin{displaymath}
L\hookrightarrow L\oplus L^{*} \hookleftarrow (L^{*})^{\rm op}
\end{displaymath}
into  the two summands are homomorphisms of Lie pseudo-bialgebras,
that is $\delta_{L\oplus L^{*}} =\delta_{L}-\delta_{ L^{*}}$.
Moreover, ${L\oplus L^{*}}$ is a quasitriangular Lie
pseudo-bialgebra.
\end{theorem}

The Lie pseudo-bialgebra ${L\oplus L^{*}}$ is called the {\it
pseudo Drinfeld double} of $L$ and is denoted by $\D L$.

\begin{proof}Let $\{e_i\}_{i=1}^n$  be an $H$-basis of $L$ and let
$\{e_i^*\}_{i=1}^n$ be the corresponding dual basis in $ L^{*}$.
Suppose, as before, that
\begin{displaymath}
[e_i \ *\  e_j]=\sum_s (h_s^{ij}\tt l^{ij}_s)\tt_H e_s,
\quad\hbox{and} \qquad [e_i^* \ * \ e_j^*]=\sum_s (f_s^{ij}\tt
g^{ij}_s)\tt_H e_s^*.
\end{displaymath}

Let $r=\sum_{i=1}^n e_i\otimes e_i^*\in L\otimes L^{*}\subset \D
L\otimes\D L$ be the canonical element corresponding to $\mathcal
I\in$Chom$(L,L)\simeq L\otimes L^{*}$ (see
Proposition~\ref{prop:ddd}), where $\mathcal I (a)=(1\tt 1)\tt_H
a$. Now, let us verify that $\delta_{L\oplus L^{*}}
:=\delta_{L}-\delta_{ L^{*}}= d (1\tt_H r)$. Using (\ref{eq:
manin-bkt}) and (\ref{jose}), we have

\vskip  .3cm

\begin{align*}
(d(1\tt_H r))*e_j&=\sum_{i,k}
\left(h_k^{ji}(l_k^{ji})_{(-1)}\right)_{(1)}\,(l_k^{ji})_{(2)}\,e_i\tt
\left(h_k^{ji}(l_k^{ji})_{(-1)}\right)_{(2)} e_i^*\nonumber\\
&\qquad
-\sum_{i,s}\left(h_i^{js}(l_i^{js})_{(-2)}{(l_i^{js})_{(-1)}}_{(-1)}\right)_{(1)}\,e_i\tt\nonumber\\
&\hskip
2cm\left(h_i^{js}(l_i^{js})_{(-2)}{(l_i^{js})_{(-1)}}_{(-1)}\right)_{(2)}
{(l_i^{js})_{(-1)}}_{(2)}\, e_s^*\nonumber\\
&\qquad+\sum_{r,i}
\left((g_j^{ir})_{(-2)}{(g_j^{ir})_{(-1)}}_{(-1)}(f_j^{ir})_{(-1)}\right)_{(1)}\,
e_i\tt \nonumber\\
&\qquad\hskip 1cm
\left((g_j^{ir})_{(-2)}{(g_j^{ir})_{(-1)}}_{(-1)}(f_j^{ir})_{(-1)}\right)_{(2)}\,(f_j^{ir})_{(2)}
{(g_j^{ir})_{(-1)}}_{(2)}e_r\nonumber\\
&=\sum_{i,r} S(f_j^{ir})\,e_i\tt S(g_j^{ir})e_r\nonumber\\
&=\delta_L(e_j).
\end{align*}
Similarly, by using Proposition~\ref{duality-thm} with (\ref{eq:
manin-bkt}), and then skew-symmetry, we get
\begin{align*}
(d(1\tt_H r))*e_j^*&=\sum_{s,i}{(l_j^{is})_{(-1)}}_{(1)}
{{(l_j^{is})_{(-2)}}_{(-1)}}_{(1)}
{(h_j^{is})_{(-1)}}_{(1)}(h_j^{is})_{(2)}
{(l_j^{is})_{(-2)}}_{(2)}\,e_s^*\nonumber\\
&\hskip 3cm\tt
{(l_j^{is})_{(-1)}}_{(2)}{{(l_j^{is})_{(-2)}}_{(-1)}}_{(2)}
{(h_j^{is})_{(-1)}}_{(2)}\, e_i^*\nonumber\\
&\qquad-\sum_{r,i} (f_i^{jr})_{(1)} {(g_i^{jr})_{(-1)}}_{(1)}
{{(g_i^{jr})_{(-2)}}_{(-1)}}_{(1)}
{(g_i^{jr})_{(-2)}}_{(2)}\,e_r\nonumber\\
& \hskip 3cm \tt (f_i^{jr})_{(2)}{(g_i^{jr})_{(-1)}}_{(2)}
{{(g_i^{jr})_{(-2)}}_{(-1)}}_{(2)}\,e_i^*\nonumber\\
&\qquad+\sum_{i,k} (f_k^{ji})_{(1)} {(g_k^{ji})_{(-1)}}_{(1)}e_i
 \tt (f_k^{ji})_{(2)}{(g_k^{ji})_{(-1)}}_{(2)}(g_k^{ji})_{(2)}e_k^*\nonumber\\
&=\sum_{i,s} S(l_j^{is}) e_s^*\tt S(h_j^{is})e_i^*\nonumber\\
&=-\sum_{s,i}  S(h_j^{si})e_s^*\tt S(l_j^{si})e_i^*\nonumber\\
&=-\delta_{L^*}(e_j^*).
\end{align*}
It remains to see that $r$ gives us a quasitriangular structure
(recall Definition \ref{def:quasi}). Using (\ref{[[r,r]]}), we
have

\begin{align*}
[[r,r]]&= \mu_{-1}^3([e_j,e_i]\otimes e^*_j\otimes e^*_i)
-\mu_{-2}^4(e_i\otimes[e_j,e^*_i]\otimes e^*_j)
-\mu_{-3}^2(e_i\otimes e_j\otimes[e^*_j,e^*_i])\nonumber\\
&=\sum_{k,j,i} \Big( (l_k^{ji})_{(2)}e_k\tt
(l_k^{ji})_{(1)}S(h_k^{ji})e_j^*\tt e_i^* \nonumber \\
&\qquad\qquad + e_i\tt {(l_i^{jk})_{(-1)}}_{(2)} e_k^*\tt
{(l_i^{jk})_{(-1)}}_{(1)} (l_i^{jk})_{(2)} S(h_i^{jk})e_j^*
\nonumber\\
&\qquad\qquad- e_i\tt( f_j^{ik})_{(2)} {(g_j^{ik})_{(-1)}}_{(2)}
e_k\tt ( f_j^{ik})_{(1)} {(g_j^{ik})_{(-1)}}_{(1)}
(g_j^{ik})_{(2)}
e_j^*\nonumber\\
&\qquad\qquad- e_i\tt (g_j^{ki})_{(1)}S(f_j^{ki}) e_k\tt
(g_j^{ki})_{(2)}e_j^*\Big).
\end{align*}
 Now, the last two terms cancels out by skew-symmetry (after interchanging the
summation  indices $j$ and $k$). Then, it is easy to see, using
skew symmetry, that\linebreak $[[r,r]]=0 $ mod $ H^+\cdot(L\tt
L\tt L)$, since
\begin{align*}
[[r,r]]&= \sum_{k,j,i}\big[ (l_k^{ji})_{(2)}e_k\tt
(l_k^{ji})_{(1)}S(h_k^{ji})e_j^*\tt e_i^* \nonumber \\
&\hskip 3.5cm + e_i\tt ((l_i^{jk})_{(-1)})_{(2)} e_k^*\tt
((l_i^{jk})_{(-1)})_{(1)} (l_i^{jk})_{(2)} S(h_i^{jk})e_j^* \big]
\nonumber\\
&= \sum_{k,j,i}\left[ \left(l_k^{ji})_{(2)}\tt
(l_k^{ji})_{(1)}S(h_k^{ji})\tt 1\right)+ \left(1 \tt
S(l_i^{jk})\tt S(h_k^{ij})\right)\right](e_k\tt e_j^*\tt e_k^*)
\nonumber\\
&= \sum_{k,j,i}\left[ \left(l_k^{ji})_{(2)}\tt
(l_k^{ji})_{(1)}S(h_k^{ji})\tt 1\right)- \left(1 \tt
S(h_k^{ji})\tt S(l_k^{ji})\right)\right](e_k\tt e_j^*\tt e_k^*)
\nonumber\\
&= \sum_{k,j,i} \left[(l_k^{ji})_{(1)}-\varepsilon(
(l_k^{ji})_{(1)})\mathbf{1}\right] \left(1\tt S(h_k^{ji})\tt
(l_k^{ji})_{(2)}\right)(e_k\tt e_j^*\tt e_k^*),
\end{align*}
and $ \left((l_k^{ji})_{(1)}-\varepsilon(
(l_k^{ji})_{(1)})\mathbf{1}\right)\in H^+$. Finally, by similar
computations, it is possible to verify that
\begin{displaymath}
\mu_2\left(e_i *\ (r+r_{21})\right) = \mu_2\Big(e_i * \big(\sum_j
e_j\otimes e_j^* + e_j^*\otimes e_j\big)\Big)=0,
\end{displaymath}
finishing the proof.\end{proof}

\vskip 1cm

\section{$\A_Y (L)$ and the annihilation algebra}\label{subanni}

\vskip .5cm

A Lie algebra is usually associated to a  Lie pseudoalgebra $L$,
that is  the annihilation algebra, see Remark 7.2 in \cite{BDK}.
Their construction, at first sight, even for the conformal case,
namely $H=\mathbb{C}[\partial]$, is not natural unless you look at
a similar notion from vertex algebra theory. In this section,
using the language of $H$-coalgebras, we will see it as a
convolution algebra of certain type, obtaining a more natural and
conceptual construction.

Here we will recall the definition of the annihilation algebra of
a pseudoalgebra.  Let $Y$ be a commutative associative
$H$-\difalg\ with a right action of $H$, and let $L$ be a Lie
$H$-\psalg. We provide $Y \tt L$ with the following structure of a
left $H$-module:
\begin{equation*}
h(x \tt a) = x h_{(-1)} \tt h_{(2)} a, \qquad h\in H, \; x\in Y,
\; a\in L.
\end{equation*}
Then define a Lie pseudobracket on $Y \tt L$ by the formula:
\begin{equation}\lbb{affps}
[(x \tt a)*(y \tt b)] = \tsum_i\, \bigl( {f_i}_{(1)} \tt
{g_i}_{(1)} \bigr) \tt_H \bigl( (x {f_i}_{(2)}) (y {g_i}_{(2)})
\tt e_i \bigr),
\end{equation}
if $[a*b]=\sum_i (f_i\otimes g_i)\otimes_H e_i$. It is easy to
check that \eqref{affps} is well defined and endows $Y \tt L$ with
the structure of a Lie $H$-\psalg. Now, we  define  the Lie
algebra
\begin{equation*}
\A_Y (L) = (Y \tt L) / H_+(Y \tt L),
\end{equation*}
with bracket
\begin{align*}
[\overline{x\tt a},\overline{y\tt b}]&=\sum_i
\varepsilon({f_i}_{(1)}) \varepsilon({g_i}_{(1)})\left(
\overline{(x{f_i}_{(2)})(y{g_i}_{(2)} )\tt
e_i}\right)\nonumber\\
&= \sum_i\overline{(x{f_i})(y{g_i} )\tt e_i}.
\end{align*}

In the case $H=\kk[\d]$, $Y=\kk[t,t^{-1}]$, $\d=-\d_t$, the Lie
$\kk[\d]$-\psalg\ $Y \tt L$, (i.e.: in the classical conformal
algebra case),  is known as an {\em affinization\/} of the
conformal algebra $L$ , (see \cite{K2}).

Now, take  $Y=X=H^*$. Recall that $X$ has a right $H$-module
structure given by
$$
\langle x\cdot h,f\rangle=\langle x,S(h)f\rangle,
$$
with $h,\,f\in H$ and $x\in  X$.
\begin{definition}\lbb{dannih}
The Lie algebra $\A(L) \equiv \A_X (L) $ is called the {\em
annihilation algebra\/} of the \psalg\ $L$.
\end{definition}

Now, we can give an interpretation of $\A_Y(L)$, as a convolution
algebra.

\begin{theorem}
Let $L$ be a free finite Lie pseudoalgebra, let $(L^* , \delta)$
be the corresponding Lie $H$-coalgebra. Then there is an
isomorphism of Lie algebras
\begin{displaymath}
\A_Y(L)\ \simeq \ \hbox{Hom}_H(L^*,Y).
\end{displaymath}
with the bracket in the  space of homomorphisms given by
\begin{displaymath}
[\, f\, , \, g\, ]= m \circ (f\otimes g) \circ \delta,
\end{displaymath}
where $m$ stands for the multiplication in $Y$.
\end{theorem}

\begin{proof}Let us define the map
\begin{align*}
\phi: (Y \tt L) / H_+(Y \tt L)&\longrightarrow
\hbox{Hom}_H(L^*,Y),\nonumber\\
\overline{x\tt a}&\longmapsto \phi(\overline{x\tt a})
\end{align*}
   such that for all $f\in L^*$
$$\phi(\overline{x\tt a})(f)=
x\cdot(l_iS(h_i))
$$
if  $f(a)=\sum_i (h_i\tt l_i)\tt_H 1$.
It is straight forward to check that this map is well defined and
$\phi(\overline{x\tt a})\in \hbox{Hom}_H(L^*,Y)$.

Let's check that  $\phi$ is injective. Consider a H-base set for
$L$, namely $\{e_i\}$. Assume that $\phi(\overline{\sum_i x_i\tt
e_i})=0$. This means that $\phi(\overline{\sum_i x_i\tt
e_i})(f)=0$ for all $f\in L^*$. Suppose that there exists
$x_{i_0}\neq 0$ and take $f\in L^*$ such that $f(e_{i})=(1\tt
1)\tt_ H \delta_{i,i_0}$. In this case $\phi(\overline{\sum_i
x_i\tt e_i})(f)=x_{i_0}=0$ which is a contradiction.

Now, consider an $H$-basis $\{f_i\}_{i=1}^n$ of $L^*$, , and
$\alpha\in$ Hom$_H(L^*,Y)$. Set $\alpha(f_i)=\sum_j x^{ij}h^{ij}$,
thus it is easy to check that
$$
\phi(\sum_{i,j} \overline{x^{ij}\tt h^{ij}e_i})(f_k)=\alpha(f_k),
$$
for all $k$, proving surjectivity.

Finally it remains to show that this is a Lie algebra
homomorphism. As always, consider $\{e_i\}$ a basis for $L$ and
$\{e_i^*\}$ the corresponding dual basis for $L^*$. Thus, if
$[e_n*e_m]=\sum_i (h^{nm}_i\tt l^{nm}_i)\tt_H e_i$, we have
\begin{align*}
\phi([\overline{x\tt e_n},\overline{y\tt e_m}])(e_k^*)&=\phi\left(
\sum_i\overline{(x h^{nm}_i)(yl^{nm}_i)\tt e_i}\right)(e_k^*)\nonumber\\
&=(x h^{nm}_k)(yl^{nm}_k).
\end{align*}
On the other side,
\begin{align*}
[\phi(\overline{x\tt e_n}),\phi(\overline{y\tt
e_m})](e_k^*)&=\left( m\circ (\phi(\overline{x\tt e_n})\tt
\phi(\overline{y\tt
e_m}))\circ \delta \right)(e_k^*)\nonumber\\
&=\sum_{s,t} m\left( \phi(\overline{x\tt e_n})(S(h^{st}_k)
e_s^*)\tt \phi(\overline{y\tt e_m})(S(l^{st}_k)
e_t^*)\right)\nonumber\\
&= \sum_{s,t} m\left( \delta_{n,s}\ (x\cdot h_k^{st})\tt
\delta_{m,t} (y\cdot l^{st}_k)\right)\nonumber\\
&= (x h^{nm}_k)(yl^{nm}_k),
\end{align*}
finishing the proof.
\end{proof}


\noindent {\bf Acknowledgement. } The  authors were supported in
 part by
Conicet, ANPCyT, Agencia C\'ordoba Ciencia and Secyt-UNC
(Argentina).



\end{document}